\documentclass[11pt]{article}
\usepackage[utf8]{inputenc}
\usepackage[british]{babel}
\usepackage{csquotes}
\usepackage{amsmath,amsthm,amssymb,amsfonts}
\usepackage{mathtools}
\usepackage{mathrsfs}

\usepackage[shortlabels]{enumitem}
\usepackage[overload]{empheq}

\usepackage[unicode,colorlinks=true,pagebackref=false,linkcolor=blue,citecolor=green,urlcolor=black]{hyperref}

\usepackage{algorithm}
\usepackage{algorithmic}
\usepackage{dsfont}
\usepackage[justification=centering,singlelinecheck=false]{caption}
\usepackage{subcaption}
\usepackage{graphicx}

\allowdisplaybreaks

\theoremstyle{plain}
\newtheorem{lemma}{Lemma}[section] 

\theoremstyle{definition}
\newtheorem{definition}[lemma]{Definition}

\theoremstyle{remark}
\newtheorem{remark}{Remark}[section] 

\theoremstyle{plain}
\newtheorem{theorem}[lemma]{Theorem}

\theoremstyle{plain}
\newtheorem{corollary}[lemma]{Corollary}

\theoremstyle{plain}
\newtheorem{assumption}[lemma]{Assumption}

\theoremstyle{plain}
\newtheorem{proposition}[lemma]{Proposition}

\theoremstyle{remark}
\newtheorem{example}{Example}

\numberwithin{equation}{section}

\def\cC{\mathcal{C}}
\def\cD{\mathcal{D}}

\def\cG{\mathcal{G}}

\def\cK{\mathcal{K}}

\def\cP{\mathcal{P}}

\def\cV{\mathcal{V}}

\def\sE{{\mathbb{E}}}

\def\sG{{\mathbb{G}}}

\def\sN{{\mathbb{N}}}

\def\sR{{\mathbb R}}
\def\sS{{\mathbb{S}}}
\def\sV{{\mathbb{V}}}

\DeclareMathOperator*{\esssup}{ess\,sup}

\title{
Learning 
Distributed  Equilibria in  Linear-Quadratic Stochastic Differential    Games: An $\alpha$-Potential Approach}
\author{Philipp Plank\thanks{Department of Mathematics, Imperial College London,  London,  UK  ({\tt p.plank24@imperial.ac.uk, yufei.zhang@imperial.ac.uk})} \and Yufei Zhang\footnotemark[1]}
\date{ }

\begin{document}

\maketitle

\begin{abstract}

We analyze independent policy-gradient (PG) learning in $N$-player linear-quadratic (LQ) stochastic differential games. Each player employs a distributed policy that depends only on its own state and updates the policy independently using the gradient of its own objective.
We establish   global linear convergence of these methods to an equilibrium by showing that the LQ game admits an $\alpha$-potential structure, with $\alpha$     determined by the degree of pairwise interaction asymmetry.
For pairwise-symmetric interactions, we   construct an affine distributed equilibrium  by minimizing the potential function and show that independent PG methods converge globally to this equilibrium, with    complexity scaling linearly in the population size and logarithmically in  the desired   accuracy. For   asymmetric interactions, we 
prove that independent projected PG algorithms converge linearly to an approximate equilibrium, with suboptimality   proportional  to the degree of asymmetry. Numerical experiments confirm the theoretical results across both symmetric and asymmetric interaction networks.
 
\end{abstract}

\medskip

\medskip



\section{Introduction}

Can multi-agent reinforcement learning (MARL) algorithms reliably and efficiently  learn Nash equilibria (NEs) in  
$N$-player non-cooperative stochastic differential games? 
These games model strategic interactions among multiple players,   each   controlling   a stochastic differential system and optimizing an objective influenced by the actions of all   players.
They arise naturally   in diverse fields, including   autonomous driving \cite{di2025alpha}, neural science \cite{carmona2023synchronization}, ecology \cite{santambrogio2021cucker}, and optimal trading  \cite{neuman2023trading}.
  A central goal is the computation of NEs, policy profiles from which no player can improve its payoff through unilateral deviation. 
In many settings, such equilibria are analytically intractable, motivating   growing interest in learning-based approaches that approximate NEs from data collected through repeated interactions with the environment.

  Despite this promise, MARL algorithms with theoretical performance guarantees for stochastic differential games are still limited,   due to three fundamental challenges inherent in multi-agent interactions. First, 
  scalability becomes a critical issue even with a moderate number of players, 
  as the complexity of the joint strategy space grows exponentially with the population size \cite{hernandez2019survey,gu2021mean}. 
 Second, each individual player faces a non-stationary environment, as other players are simultaneously learning and adapting their policies. The analysis of MARL algorithms thus necessitates novel game-theoretic techniques beyond the single-agent setting.  
   Third, differential games typically involve continuous time and continuous state or action spaces, requiring the use of function approximation for policies. The choice of policy parameterization can significantly affect both the efficiency and convergence of MARL algorithms.

As an initial step to tackle the aforementioned challenges, 
this work   investigates the convergence of   policy gradient (PG) algorithms
for  $N$-player
linear-quadratic (LQ)    games. 
LQ games 
  play  a fundamental role in dynamic game theory, and   serve as benchmark problems for examining the performance of MARL algorithms \cite{mazumdar2020policy,hambly2023policy,hosseinirad2026linear,plank2025policy}.
  Despite their relative tractability, existing work shows that   gradient-based learning algorithms may fail to converge to NEs \cite{mazumdar2020policy}, or converge only under restrictive contraction conditions that do not scale well with the population size or the time horizon  \cite{hambly2023policy,plank2025policy}.

\paragraph{Independent learning with distributed policies.}
To ensure scalability, we adopt \emph{distributed} (also known  as decentralized) policies and focus on \emph{independent} PG methods. A distributed policy 
means that each player bases its feedback control solely on its own state, without considering the states of other players. This approach reduces the need for inter-player communication and simplifies the policy parameterization.  Independent PG algorithms further assume that each player updates its policy   independently and concurrently by following the gradient of its own objective. Combining independent learning with distributed policies ensures that the computational complexity of the algorithm scales \emph{linearly} with the number of players.

Although distributed policies are widely used in  MARL \cite{hernandez2019survey}, the existence and characterization of NEs in the resulting $N$-player games have not been rigorously studied,
and no theoretical guarantees currently exist for the corresponding PG methods.
An   exception arises in mean field games, where players are assumed to   interact \emph{symmetrically} and \emph{weakly} through empirical measures. Under this assumption, approximate distributed NEs for $N$-player games can be constructed via a limiting continuum model as $N\to\infty$ (see e.g., \cite{carmona_probabilistic_2018_I, lauriere2022learning, plank2025policy}). 
However, 
in many realistic settings,  players are not exchangeable and instead   interact   only with  subsets of players  specified  by   an underlying network. Moreover, the influence of each player on others may not vanish as the population size grows \cite{di2025alpha}. 
This motivates the study of general interaction structures  under which distributed NEs can be   characterized and learned.

\paragraph{Our contributions.}
 This work provides non-asymptotic performance guarantees for independent PG algorithms in approximating distributed NEs in a class of finite-horizon  $N$-player LQ games, 
 inspired by flocking models \cite{lacker2022case,guo2025alpha} and opinion formation models \cite{bindel2015bad}. 
In this game, each player $i$ chooses a distributed policy to linearly control the drift of its state dynamics and   minimizes a 
cost that is quadratic in its own control and in the states of all players, 
 with the influence of the product of players
  $j$ and $k$'s states  weighted by   $Q^i_{k,j}$ (see \eqref{eq:state_dynamics}-\eqref{eq:cost_generalsetting} for precise definitions).

We analyze this LQ game  and  the  associated learning algorithms by extending the $\alpha$-potential game framework developed in \cite{guo2025alpha, guo2025markov, guo2025distributed} to closed-loop games with distributed policies. In an $\alpha$-potential game,
when a player changes its strategy, the resulting change in its objective aligns with  the change in an $\alpha$-potential function up to an error  $\alpha$.

\begin{itemize}
    \item 
    We prove that the LQ game is an $\alpha$-potential game, with an $\alpha$-potential function formulated as a distributed LQ control problem. The parameter $\alpha$ is explicitly characterized by the product of the constant
$C_Q=\max_{1\le i\le   N}\sum_{j=1}^N  |Q^i_{i,j}-Q^j_{j,i}|$
    and   bounds on the first and second moments of the state dynamics induced by admissible policies (Proposition \ref{prop:Phi_is_alpha_potential}). 
\item When $C_Q = 0$, i.e., the interaction is  \emph{pairwise symmetric}, we prove that the game admits an   affine equilibrium,
given by the minimizer of the  potential function 
(Theorem \ref{theorem:optimalpolicy}). In this equilibrium, each player’s action depends linearly on the deviation of its own state from its expected value, adjusted by a time-dependent intercept term. 

We further prove that an independent PG algorithm, in which each player updates the slope and intercept parameters of its policy independently via gradient descent, converges globally to the NE policy at a linear rate. 
This implies that the algorithm’s complexity to achieve an error of $\varepsilon$  scales \emph{linearly} with   the population size  $N$  and  with  $\log(1/\varepsilon)$. 
The proof leverages the potential structure by interpreting the   learning algorithm  as a PG method applied to the potential function and by exploiting the geometry of the potential function with respect to the policy profiles (Propositions \ref{prop:graddom_K} and \ref{prop:strongconvex_lipschitz_G}).

\item 

For general interactions with  $C_Q>0$, we characterize affine equilibria through a forward-backward ordinary differential equation (ODE) system (Theorem \ref{theorem:optimalpolicy_asymmetric}) and provide sufficient conditions for the existence of an NE  (Proposition \ref{prop:ODE_system_asym_sol}). We further prove that an independent projected PG algorithm converges linearly to an approximate NE, with the approximation error scaling proportionally with the constant $C_Q$ (Theorem \ref{theorem:convergence_asymmetric_all}). 
The proof relies on interpreting the learning algorithm as a biased PG method for the $\alpha$-potential function and on quantifying the bias through a stability analysis of the associated ODEs.

\end{itemize} 

To the best of our knowledge, this work provides the first global convergence result for independent PG  methods in continuous-time LQ games with general non-exchangeable interactions.  

\paragraph{Challenges beyond existing works.}
The characterization of distributed NEs   requires novel techniques beyond standard LQ games with full-information policies. When  $C_Q=0$, we construct  an NE   by minimizing the potential function. This reduces to a distributed control problem, which we solve using a dynamic programming approach as in \cite{jacksonlacker2025approximately}. Specifically, we lift the problem to a McKean-Vlasov control problem over the space of product measures and solve the resulting infinite-dimensional Hamilton-Jacobi-Bellman (HJB) equations.
In the general case with $C_Q>0$, we characterize the NE   by a coupled system, consisting of $N$ backward HJB equations, which determine each player’s optimal policy given other players' state distributions, and $N$ forward Fokker-Planck equations, which describe  the state evolution   under these policies. 

Analyzing independent PG  methods
for LQ games also presents challenges beyond those in prior studies of discrete-time Markov ($\alpha$-)potential games with finite state and action spaces (see e.g.,
\cite{ding2022independent, fox2022independent,
leonardos2022global,
zhang2023markov, jordan2024independent,  maheshwari2025independent, guo2025markov}). First, existing analyses typically quantify algorithm convergence in terms of the cardinality of the state and action spaces. This approach is   inapplicable to  continuous-time LQ games, which involve continuous state and action spaces and policies that vary continuously with time.  
In fact, even for discrete-time LQ games with distributed policies, existing results only establish asymptotic convergence of PG methods to saddle points, rather than to NEs, and do not provide convergence rates \cite{hosseinirad2026linear}. We substantially strengthen these results by proving linear convergence to NEs. Achieving this requires a more refined landscape analysis that characterizes the geometry of the cost functions with respect to the policy parameters, 
which is enabled by an appropriate choice of policy parameterization (Remark \ref{rmk:policy_parameterization}).

Second, analyzing algorithm convergence for  $\alpha>0$ is technically involved. 
In this setting, the gradient of each player’s objective constitutes a biased gradient of the $\alpha$-potential function, with the bias scaling with $\alpha$. Moreover, as shown in Proposition \ref{prop:Phi_is_alpha_potential}, when $C_Q>0$,  the magnitude of $\alpha$  depends critically on the regularity of the policies. It is therefore essential to ensure that this policy regularity is preserved uniformly throughout the learning process and to control how the bias propagates across iterations.

\paragraph{Notation.}
For a Euclidean space $E$, we denote by  $\langle \cdot, \cdot \rangle$ the standard scalar product, and by $\lvert \cdot \vert$  the induced norm. For a matrix $A \in \mathbb{R}^{d \times k}$, we denote by $A^\top$   the transpose of $A$, and 
by $\lambda_{\mathrm{max}}(A)$ and $\lambda_{\mathrm{min}}(A)$   the largest and smallest eigenvalue of $A$, respectively. 
For each $n\in \mathbb N$,  we denote by $\mathbb S^n_{\ge 0}$   the space of $n\times n$  symmetric positive semidefinite matrices. 

We  introduce the following function spaces: 
$\mathcal{C}([0,T], E)$ is the space of continuous functions $f:[0,T]\to E$; 
$L^1([0,T], E)$ (resp.~$L^\infty([0,T], E)$) is  the space of Borel measurable functions $f: [0,T] \to E$ such that 
$\lVert f \rVert_{L^1} \coloneqq \int_0^T{\lvert f(t) \rvert} \, dt  < \infty$ (resp.~$\lVert f \rVert_{L^\infty} \coloneqq \esssup_{t \in [0,T]}{\lvert f(t) \rvert} < \infty$); 
and 
$L^2([0,T], E)$ is the space of Borel measurable functions $f: [0,T] \to E$ with finite $L^2$-norm $\|\cdot\|_{L^2}$ induced by the canonical inner product $\langle \cdot, \cdot \rangle_{L^2}$.

For any square-integrable random variable $X$ defined on a probability space $(\Omega, \mathcal{F}, \mathbb{P})$, we denote by $\mathbb{E}[X]$ its expectation and by $\mathbb{V}[X]$ its variance.

\section{Distributed equilibria
for stochastic LQ  games}

\label{section:distributed_equi}

This section introduces a class of LQ stochastic differential games with distributed policies. In these games,   each player controls the drift of a linear stochastic differential equation through a policy that depends solely on its own state, and aims to minimize a quadratic cost functional that depends on both its individual control and the joint state of all players.

\subsection{Mathematical setup}

Let $T > 0$ be a given terminal time,   $N \in \mathbb{N}$ be the number of players,  and $I \coloneqq \{1,2,\dots,N \}$ be the index set of  players. Let $(\Omega,\mathcal{F},\mathbb{P})$ be a probability space which supports   an $N$-dimensional Brownian motion $B = (B^1, \dots, B^N)^\top$, 
and mutually independent random variables 
$(\xi^i)_{i\in I}$. 
The random variables $\xi^i$ and $B^i$ represent the initial condition and the idiosyncratic noise of the state process for player $i$.
We assume that  $\xi^i$, $i\in I$, has a positive variance  $\mathbb{V}[\xi^i]  > 0$.

 We consider a stochastic differential game with  distributed policies  defined as follows. 
For each $i\in I$, 
player $i$ chooses a policy from the set  $\tilde{\mathcal V}^i \subset \mathcal V^i$, 
where $\mathcal V^i$  is the set of  all measurable functions 
$u^i:[0,T]\times \mathbb R\to  \mathbb R$
such that the associated state dynamics 
\begin{equation}
    \label{eq:state_dynamics}
    d X_t^{i} = u^i(t, X_t^{i}) \, dt + \sigma_t^i \, d B_t^{i}, \quad t \in [0,T]; \quad X_0^i = \xi^i
\end{equation}
admits a unique strong solution 
$X^{u^i,i}$ satisfying 
$\mathbb E[\sup_{t\in [0,T]}|X^{u^i,i}_t|^2]
+\mathbb E[\int_0^T  |u^i(t,X^{u^i,i}_t)|^2 \, d t]<\infty$.
The precise definition of $\tilde{\mathcal V}^i$
  will be provided later in the appropriate context. 
  We say  
any policy in $\mathcal V^i$ is  distributed since it depends only on player $i$'s own state.
Let $\tilde{\mathcal{V}}\coloneqq \prod_{i \in I} \tilde{\mathcal{V}}^i$ be  
  the set of admissible policy profiles of all players.
Player $i$ aims to  minimize the following cost functional $J^i:\tilde{\cV}\to \sR$:
\begin{equation}
    \label{eq:cost_generalsetting}
    J^i(u) \coloneqq \mathbb{E}\left[\int_0^T \bigg(  \lvert u^i(t,X^i_t) \rvert^2 + X_t^\top Q^i X_t \bigg) \, dt + \gamma^i \lvert X_T^i - d^i \rvert^2 \right], 
\end{equation}
where 
$X\coloneqq (X^1,\ldots, X^N)^\top $
 depends on $u $ through the dynamics 
\eqref{eq:state_dynamics}.
To simplify the notation,   this explicit dependence is omitted whenever there is no risk of confusion.
We assume 
$\sigma^i\in \mathcal C([0,T],\mathbb R)$,
$Q^i\in \mathbb S^N_{\ge 0}$, $\gamma^i\ge 0$ and $d^i\in \mathbb R$.

We denote by $\sG  = (I, (J^i)_{i \in I}, \tilde{\cV})$ the game defined by \eqref{eq:state_dynamics}-\eqref{eq:cost_generalsetting}, emphasizing its dependence on the admissible policy profiles $\tilde{\mathcal{V}}$. 
For each $i\in I$,
we denote  by  
$\tilde {\mathcal V}^{-i}=\prod_{j\in I\setminus \{i\}}\tilde {\mathcal V}^j$ the set of policy profiles of
  all players except player $i$. 
When $\tilde {\mathcal V}^i={ \mathcal V}^i$, we write $ \mathcal V =\tilde {\mathcal V}$ and ${\mathcal V}^{-i}=\tilde{\mathcal V}^{-i} $.

\begin{remark}
 For ease of exposition, we have assumed that each player has one-dimensional state and control processes in \eqref{eq:state_dynamics}, and that the terminal cost of  \eqref{eq:cost_generalsetting}  depends only on player $i$'s   state. All results can be   extended to multi-dimensional state  dynamics of the form  
 $ 
 d X_t^{i} =(A^i_t X_t^{i}+D^i_t u^i(t, X_t^{i})) \, dt + \sigma_t^i \, d B_t^{i},
 $
and to terminal costs that depend on all players' joint states.
   
\end{remark}
 
Now we  introduce the definition of an $\varepsilon$-Nash equilibrium ($\varepsilon$-NE) for the game.

\begin{definition}
\label{def:NE}
   A policy profile  $u^* = (u^{*,i})_{i \in I} \in \tilde{\mathcal{V}}$ is  called a (distributed) $\varepsilon$-Nash equilibrium in the policy class $\tilde \cV$
    if for all $i \in I$ and ${u}^i \in \tilde{\cV}^i$,
    \begin{equation*}
        J^i(u^{*,i}, u^{*,-i}) \le J^i({u}^i, u^{*,-i}) + \varepsilon.
    \end{equation*}
  When $\varepsilon=0$, $u^*$ is called  
     a Nash equilibrium of the game. 
\end{definition}

The   equilibrium concept introduced in Definition~\ref{def:NE} is a distributed equilibrium, since each player only optimizes over distributed policies  \cite{cirant2025non}. Clearly,    $\varepsilon$-NEs
of the game depend 
on each player's  admissible policies.

\subsection{$\alpha$-potential function}

The main objective of this paper is to analyze the convergence of   gradient-based learning algorithms in which all players simultaneously update their policies based on the gradients of their individual objectives. The primary analytical tool for characterizing approximate NEs  and analyzing learning algorithms is the concept of an $\alpha$-potential function introduced in \cite{guo2025markov, guo2025alpha}.

 \begin{definition} \label{def:alphapotentialgame}
   Given the set  of policy profiles $\tilde{\cV}$,
    a   function $\Phi: \tilde{\cV} \to \mathbb{R}$  is said to be an $\alpha$-potential function of
the game $\sG = (I, (J^i)_{i \in I}, \tilde{\cV})$ if 
for all $i\in I$,
$u^i,\tilde u^i\in \tilde{\cV}^i$ and $u^{-i}\in \tilde{\cV}^{-i}$, 
    \begin{equation*}
        \lvert  [ J^i( \Tilde{u}^i, u^{-i} ) - J^i( u^i, u^{-i} )]
        -[\Phi(\Tilde{u}^i, u^{-i} ) - \Phi( u^i, u^{-i} ) ]\rvert \le \alpha.
    \end{equation*}
When such a $\Phi$ exists,    the game $\sG$ is called an $\alpha$-potential game.
In the case   
      $\alpha = 0$,
      $\Phi$ is called potential function and the game is called a potential game.  
\end{definition}

An $\alpha$-potential function ensures that if a player unilaterally deviates from its strategy, the resulting change in that player’s objective function coincides with the change in the $\alpha$-potential function up to an error   $\alpha$. As shown in 
\cite[Proposition 2.1]{guo2025towards}, a minimizer of an $\alpha$-potential function  $\Phi$ is an $\alpha$-NE of the game $\sG$.

We now construct an $\alpha$-potential function for the game $\sG$. This construction exploits the distributed structure of the game and extends the approach for dynamic potential functions in \cite{guo2025towards} to the more general case of $\alpha >0$.
 To this end, define $Q=(Q_{i,j})_{i,j\in I}\in \sR^{N\times N}$ by 
 \begin{equation}
    \label{eq:Q_matrix}
    Q_{i,j} \coloneqq 
    \begin{cases}
        Q_{i,i}^i + \frac{1}{2} \sum_{j \in I\setminus \{i\}} (Q_{i,j}^i - Q_{j,i}^j),
        & i=j,
        \\
        Q^i_{i,j}, &i\not =j,
    \end{cases}
\end{equation}
  where $Q_{i,j}^i$  
is the $(i,j)$-th entry of the matrix $Q^i$.
Let  $\Lambda \coloneqq \mathrm{diag}(\gamma^1,\dots, \gamma^N) \in \mathbb{R}^{N \times N}$, $d \coloneqq (d^1, \dots, d^N)^\top \in \mathbb{R}^N$,
and define the function
$\Phi:\cV\to \sR$ by 
\begin{equation}
    \label{eq:Phi_general_def}
    \Phi( u ) \coloneqq \mathbb{E}\left[ \int_0^T \left( \sum_{i=1}^N |u^i(t,X^i_t)|^2 + X_t^\top Q X_t \right) \, dt + (X_T - d)^\top \Lambda (X_T - d) \right].
\end{equation}
The following proposition shows that $\Phi$  is an $\alpha$-potential function and derives an explicit upper bound on $\alpha$ in terms of the interaction matrices $(Q^i)_{i \in I}$ and the policy class $\tilde{\mathcal{V}}$.

  \begin{proposition}
    \label{prop:Phi_is_alpha_potential}
    The function $\Phi$ defined in \eqref{eq:Phi_general_def} is an $\alpha$-potential function of the game $\sG = (I, (J^i)_{i \in I}, \tilde{\cV})$, where
    \begin{equation}
    \label{eq:alpha_bound_general}
        \alpha \coloneqq \left(\max_{i \in I} \sup_{u^i \in \tilde{\cV}^i}  \|{\mathbb{V}[X^{u^i,i}]}\|_{L^1} + 3 \max_{i \in I} \sup_{u^i \in \tilde{\cV}^i}  {\|\mathbb{E}[ X^{u^i,i}]\|^2_{L^2}} \right) C_Q,
    \end{equation}
    and
    \begin{equation}
        \label{eq:cw}
        C_Q \coloneqq \max_{i \in I} \sum_{j \in I\setminus \{i\}}  \lvert Q_{i,j}^i - Q_{j,i}^j \rvert.
    \end{equation}
    In particular, when $C_Q=0$,
    $\Phi$ is a potential function of the game $\sG = (I, (J^i)_{i \in I}, {\cV})$.
\end{proposition}

\begin{remark}
    The upper bound
    \eqref{eq:alpha_bound_general}
    on $\alpha$  depends on the magnitude of asymmetric interactions between any two players in the dynamic game, as characterized by the interaction weights $(Q^i)_{i \in I}$, and the moments of the state processes induced by policies in $\tilde{\mathcal{V}}$.

The constant $C_Q$ in \eqref{eq:cw} depends only on the interaction structure of the game $\sG$.
It is clear that 
 $C_Q=0$ if $Q^i_{i,j}=Q^j_{j,i}$, i.e., the pairwise symmetric interaction.
When the interaction is asymmetric, 
it has been shown in the literature that, for suitably structured network games, $C_Q$ vanishes as the number of players tends to infinity. This occurs, for instance, when the interaction weights $Q^i_{i,j}$ and $Q^i_{j,i}$ decay  sufficiently fast with respect to the graph distance between two players  \cite{guo2025alpha, guo2025distributed}, or when the interaction weights  are generated by sufficiently dense or sufficiently sparse random networks \cite{rokade2025asymmetric}.

\end{remark}

Throughout this paper, we impose the   condition that the $\alpha$-potential function $\Phi$ is convex, and hence admits a minimizer in $\mathcal{V}$.

\begin{assumption}
    \label{assumption:Q_positivesemidef}
    $Q_{\mathrm{sym}} \coloneqq (Q + Q^\top)/2 \in \sS^N_{\ge 0}$, with  $Q$   defined in \eqref{eq:Q_matrix}.
\end{assumption}

Below we verify Assumption  \ref{assumption:Q_positivesemidef}
for specific interaction structures. 

\begin{example}
   \label{example:nr1}
   
Suppose the cost functional \eqref{eq:cost_generalsetting} is of the following form (see  \cite{bindel2015bad, guo2025alpha, rokade2025asymmetric}):  
    \begin{equation}
        \label{eq:cost_crowd_motion_lq}
        J^i(u) = \mathbb{E}\left[\int_0^T \bigg( \lvert u^i(t,X_t^i) \rvert^2 + \sum_{j \in I\setminus \{i\}} \omega_{i,j} \lvert X_t^i -  X_t^j \rvert^2 \bigg) \, dt + \gamma^i \lvert X_T^i - d^i \rvert^2 \right], 
    \end{equation}
    where $ w_{i,j}\ge 0$ for all $i,j \in I$.  In this case, the matrix $Q \in \mathbb{R}^{N \times N}$ in  \eqref{eq:Q_matrix}  and 
    the constant  $C_Q$ in  \eqref{eq:cw} are 
    given by 
$$
Q_{i,j}=\begin{cases}
  \frac{1}{2}  \sum_{l \in I\setminus \{i\}}( {\omega_{i,l} + \omega_{l,i}}), & i=j,
   \\
     -\omega_{i,j}, & i\not = j,
\end{cases}
\qquad 
C_Q = \max_{i \in I} \sum_{j = 1}^N \lvert \omega_{i,j} - \omega_{j,i} \rvert.
$$
Note that Assumption \ref{assumption:Q_positivesemidef} always holds,  since $Q_{\mathrm{sym}}$ is symmetric, 
    diagonally dominant, and has  non-negative diagonal entries.

\end{example}

\begin{example}
\label{ex:mean_flocking}

Suppose the cost functional \eqref{eq:cost_generalsetting} is of the following form  (see \cite{carmona2013mean, lacker2022case}):
    \begin{equation*}
        J^i(u) = \mathbb{E}\left[\int_0^T \bigg( \lvert u^i(t,X_t^i) \rvert^2 + \bigg\lvert X_t^i - \sum_{j \in I\setminus \{i\}} \omega_{i,j} X_t^j \bigg\rvert^2 \bigg) \, dt + \gamma^i \lvert X_T^i - d^i \rvert^2 \right],
    \end{equation*}
    where $ w_{i,j}\in \mathbb R$ for all $i,j \in I$.
    In this case, the matrix $Q \in \mathbb{R}^{N \times N}$ in  \eqref{eq:Q_matrix}  and 
    the constant  $C_Q$ in  \eqref{eq:cw} are 
    given by  
    $$
Q_{i,j}=\begin{cases}
   1 + \frac{1}{2} \sum_{l \in I\setminus \{i\}} (\omega_{l, i} - \omega_{i,l}), & i=j,
   \\
     -\omega_{i,j}, & i\not = j,
\end{cases}
 \qquad
     C_Q = \max_{i \in I} \sum_{j = 1}^N \lvert \omega_{i,j} - \omega_{j,i} \rvert.$$
Assumption \ref{assumption:Q_positivesemidef} may not    hold   even if all weights $(w_{i,j})_{i,j\in I}$ are non-negative.  
A sufficient condition for  Assumption \ref{assumption:Q_positivesemidef} 
 is that 
 $ \sum_{j\not =i} \omega_{i,j}<1$ for all $i\in I$ and all weights $(w_{i,j})_{i,j\in I}$ are non-negative. 
\end{example}

According to Proposition~\ref{prop:Phi_is_alpha_potential}, the magnitude of $\alpha$ for the game $\sG$ depends on the interplay between the asymmetry of the interaction weights and the complexity of the policy class. In Section~\ref{section:symmetric}, we analyze games with symmetric interactions and general policies, and in Section~\ref{section:asymmetric}, games with asymmetric interactions and policies whose states have bounded second moments.

\section{Learning NEs under symmetric interactions  
}
\label{section:symmetric}

This section considers the LQ game $\sG= (I, (J^i)_{i \in I}, {\cV})$ with symmetric interactions. In this case, the game is a potential game with a  potential function $\Phi$ defined in \eqref{eq:Phi_general_def}. We leverage this potential structure to construct an NE of  the game and to establish the convergence of independent learning algorithms.

\subsection{Characterization of distributed NEs}

More precisely, we impose the following condition throughout this section.  

\begin{assumption}
    \label{assumption:Qij_symmetry}
      $Q_{i,j}^i = Q_{j,i}^j$ for all $i,j \in I$. In particular, $Q_{\mathrm{sym}}=Q$  with $Q$ defined in \eqref{eq:Q_matrix}.
\end{assumption}

Under Assumption \ref{assumption:Qij_symmetry},
Theorem \ref{theorem:optimalpolicy}  constructs   an NE of the game 
$\sG$   via  the minimizer of the potential function $\Phi$.
Minimizing $\Phi$ over $\cV$
 corresponds to a LQ distributed control problem, whose optimal policy can be obtained through the following  ODE system: 
for all $i \in I$ and $t \in [0,T]$,
\begin{subequations}
\label{eq:ODEs_valuefunction}
    \begin{align}[left = \empheqlbrace\,]
        &\frac{\partial P_t^i}{\partial t} - (P_t^i)^2 + Q_{i,i} = 0, \quad P_T^i = \gamma^i,
        \label{eq:P_symmetric}
        \\
            &\frac{\partial \Psi_t}{\partial t} - \Psi_t^\top \Psi_t + Q = 0, \quad \Psi_T = \Lambda,
            \label{eq:Psi_symmetric}\\
            &\frac{\partial \zeta_t^i}{\partial t} - \sum_{j = 1}^N \zeta_t^j \Psi_t^{j,i} = 0, \quad \zeta_T^i = - d^i \gamma^i.
            \label{eq:zeta_symmetric}
    \end{align}
\end{subequations}

\begin{theorem}
    \label{theorem:optimalpolicy}
    Suppose Assumption \ref{assumption:Q_positivesemidef} holds.
    For all $i \in I$, the system \eqref{eq:ODEs_valuefunction}  admits  a  unique solution $P^i, \zeta^i \in \mathcal{C}([0,T], \mathbb{R})$ and $\Psi \in \mathcal{C}([0,T],\mathbb{R}^{N \times N})$. Define $u^{*} =(u^{*,i})_{i\in I}\in \mathcal{V}$ by
  \begin{equation}
        \label{eq:optimal_policy}
        u^{*,i}(t,x) \coloneqq K_t^{\Phi, *,i} (x - \mu_t^{*,i}) + G_t^{\Phi, *,i},  
        \quad \forall 
        (t,x)\in [0,T]\times \sR, \quad \forall i\in I,
          \end{equation}
    where 
    $     K_t^{\Phi, *,i} \coloneqq - P_t^{i}$, $G_t^{\Phi, *,i} \coloneqq - [\Psi_t \mu_t^*]_{i} - \zeta_t^i$,
    and $\mu^{*,i}_t \coloneqq \sE[X^{u^{*,i},i}_t]$ satisfying 
    \begin{equation*}
        \frac{\partial \mu_t^{*,i}}{\partial t} = - [\Psi_t \mu_t^*]_i - \zeta_t^i, \quad t \in [0,T]; \quad \mu_0^{*,i} = \mathbb{E}[\xi^i].
    \end{equation*} 
Then $u^*$ is a  minimizer of $\Phi:\cV\to \sR$.

Furthermore, if 
Assumption \ref{assumption:Qij_symmetry} holds, 
    then   $u^*  $ is an NE of the game $\sG = (I, (J^i)_{i \in I},  {\cV})$. 
    
\end{theorem}

The proof proceeds by minimizing $\Phi$ using the dynamic programming approach as in \cite{jacksonlacker2025approximately}, which involves lifting the problem to a McKean-Vlasov   control problem over the space of product measures and solving the resulting infinite-dimensional HJB equations using the solutions of \eqref{eq:ODEs_valuefunction}.

\subsection{Independent policy gradient algorithm and its convergence}

Motivated by the NE policy in Theorem~\ref{theorem:optimalpolicy}, this section proposes and analyzes a policy gradient algorithm for the game $\sG$, in which, at each iteration, players simultaneously update their policies by performing gradient descent on their individual objective functions.

\subsubsection{Policy gradient algorithm}
In light of Theorem~\ref{theorem:optimalpolicy}, we consider affine policies that depend on each player’s state mean. 
For each $i\in I$, consider the following parameter space for player $i$:
$$
\mathcal{K}^i\times \cG^i \coloneqq L^2([0,T], \mathbb{R})\times L^2([0,T], \mathbb{R}), 
$$
and  
  the following policy  space   $\cV_{\mathrm{aff}}^i$  parameterized by 
  $\mathcal{K}^i\times \cG^i$:
\begin{equation}
    \label{eq:policy_parameterization}
     \mathcal V_{\mathrm{aff}}^i  \coloneqq 
     \left\{ u_{\theta^i} \in \cV^i 
     \,\middle\vert\, 
    \begin{aligned}
    &u_{\theta^i}(t,x)= K_t^i (x - \sE[X^{u_{\theta^i},i}_t]) + G^i_t,
    \quad \forall (t,x)\in [0,T]\times \sR,
    \\ 
     &
     \textnormal{$X^{u_{\theta^i},i}$ satisfies 
     \eqref{eq:state_dynamics} with $u^i=u_{\theta^i}$, and $ \theta^i = \big(K^i,G^i\big) \in \mathcal{K}^i\times \cG^i$}  
     \end{aligned}
     \right\}.
\end{equation}
  We  identify $u_{\theta^i}\in \cV_{\mathrm{aff}}^i$  
 with its parameter $ \theta^i = \big(K^i,G^i\big) \in \mathcal{K}^i\times \cG^i$,
 and the joint policy space  
 $\cV_{\mathrm{aff}}$ with 
 the joint parameter space
 $\mathcal{K} \times \mathcal{G} \coloneqq \prod_{i = 1}^N \mathcal{K}^i \times \prod_{i = 1}^N \mathcal{G}^i \cong L^2([0,T], \mathbb{R}^N) \times L^2([0,T], \mathbb{R}^N)$. 
For any $\theta^i \in \mathcal{K}^i \times \cG^i$, we denote by $X^{\theta,i}$ player $i$'s state process   satisfying \eqref{eq:state_dynamics} under the   policy $u_{\theta^i}$. 

 \begin{remark}
 \label{rmk:policy_parameterization}
  Although a policy in $\mathcal{V}_{\mathrm{aff}}^i$ includes explicit feedback with respect to each player’s state mean, this measure dependence can be viewed as an additional time-dependent component
  that is determined by the parameter $\theta^i$. As will be shown in Proposition~\ref{prop:potential_decomposition}, such a policy parameterization separates the contributions of $K^i$ and $G^i$ in the cost functional. This  property is  essential for the convergence analysis of the   policy gradient algorithms.
   
 \end{remark}
  
Given the    policy class $\mathcal{V}_{\mathrm{aff}}$, each player independently and simultaneously performs gradient descent on its cost functional to update its policy.
Specifically,  let $(K^{(0)}, G^{(0)}) = (K^{(0),i}, G^{(0),i})_{i\in I}\in \cK\times \cG$
be the initial parameter profile, 
and for each $i\in I$, 
let $\eta^i_K>0$ and $\eta^i_G>0$
be the learning rates 
used by player $i$ to update the parameters $K^i$ and $G^i$, respectively. 
For each iteration $\ell\in \sN_0\coloneqq \sN\cup \{0\}$,
given the parameter profile 
$\theta^{(\ell)}\coloneqq (K^{(\ell)}, G^{(\ell)})$,
consider the following gradient descent update: for all $i\in I$,  
\begin{equation}
    \label{eq:potential_GD_update}
    \begin{split}
        &K_t^{(\ell+1),i} \coloneqq K_t^{(\ell),i} - \eta_K^i (\nabla_{K^i} J^i(\theta^{(\ell)} ) )_t (\vartheta_t^{(\ell),i})^{-1}, \quad t\in [0,T], \\
        &G_t^{(\ell+1),i} \coloneqq G_t^{(\ell),i} - \eta_G^i (\nabla_{G^i} J^i(\theta^{(\ell)} ))_t, \quad t\in [0,T],
    \end{split}
\end{equation}
where 
$J^i(\theta^{(\ell)})$
is the cost functional \eqref{eq:cost_generalsetting}
evaluated  at the policy profile 
$  u_{\theta^{(\ell)}}$,  
$\vartheta_t^{(\ell),i}=\sV[X^{\theta^{(\ell)},i}_t]$ is the variance of player $i$'s state process  
controlled by  the policy $u_{\theta^{(\ell),i}}$,
and 
$\nabla_{K^i} J^i$ and $\nabla_{G^i} J^i$ are the (Fr\'echet) derivative of $J^i$ with respect to $K^i$ and $G^i$, respectively.

The policy update \eqref{eq:potential_GD_update} extends the single-agent policy gradient algorithm in \cite{giegrich2024convergence}
to the   multi-agent    setting. 
The gradients $\nabla_{K^i} J^i$ and $\nabla_{G^i} J^i$ can be expressed analytically in terms of the model coefficients, as stated in the following lemma. When the model coefficients are unknown, these gradients can be estimated using zeroth-order optimization methods based on trajectories of player $i$’s state and cost \cite{berahas2022theoretical}.

\begin{lemma}
\label{lemma:cost_gradients}
   For all $i\in I$,  $\theta=(K,G)\in \cK\times \cG$ and $t\in [0,T]$,
   $(\nabla_{K^i} J^i(\theta))_t = 2 (P^{K,i}_t + K^{i}_t) \vartheta^{K,i}_t$,
   and 
   $(\nabla_{G^i} J^i(\theta))_t = 2 (G^{i}_t + \Xi^{G,i}_t)$, 
where $P^{K,i}\in \cC([0,T], \sR)$ satisfies
\begin{equation}
    \label{eq:PK_ODE}
    \frac{\partial}{\partial t} P_t^i + 2 K_t^{i} P_t^i + (K_t^{i})^2 + Q_{i,i}^i = 0, \quad t \in [0,T]; \quad P_T^i = \gamma^i,
\end{equation}
$
    \Xi_t^{G,i} \coloneqq \int_t^T \left[ Q^i \mu_s^{G} \right]_i \, ds + \gamma^i (\mu_T^{G,i} - d^i),
    $ 
   $\mu^{G,i}_t \coloneqq \sE[X^{\theta, i}_t]$, 
and   $\vartheta^{K,i}_t \coloneqq \sV[X^{\theta, i}_t]$.

\end{lemma}

Algorithm~\ref{algo} summarizes the   policy gradient algorithm for the case with symmetric interactions.  

\begin{algorithm}[H]
\caption{Independent Policy Gradient Learning for Symmetric Interactions}
\label{algo}
\begin{algorithmic}[1]
\STATE \textbf{Input:} 
Initial parameters $(K^{(0),i}, G^{(0),i})_{i\in I}$,
and learning rates $ \eta_K^i, \eta_G^i\in (0,\infty)$ for all $i\in I$. 
\FOR{$\ell = 1, 2, \dots$}
    \STATE 
    Obtain the updated parameters  $ (K^{(\ell),i},G^{(\ell),i} )_{i\in I}$ by \eqref{eq:potential_GD_update}. 
\ENDFOR
\end{algorithmic}
\end{algorithm}
 
\subsubsection{Convergence analysis}

This section establishes the linear convergence of Algorithm~\ref{algo} to the NE given in Theorem~\ref{theorem:optimalpolicy}. The key observation is that, under the policy parameterization \eqref{eq:policy_parameterization}, each player’s cost functional and the potential function can be decomposed into two terms that depend only on $K$ and $G$, respectively. Such a decomposition enables separate convergence analyses for $K$ and $G$.

To see it, observe that by  restricting   to the policy space  $ \mathcal{V}_{\mathrm{aff}}$, 
  the cost functional \eqref{eq:cost_generalsetting} takes the following form: 
for each $\theta =(K^i,G^i)_{i\in I}\in \cK\times \cG$,  
\begin{equation}
    \label{eq:cost_affine_param}
    J^i(\theta) = \mathbb{E}\left[\int_0^T \left(  \lvert K_t^i (X_t^{\theta,i} - \mathbb{E}[X_t^{\theta, i}]) + G_t^i \rvert^2 + (X_t^\theta)^\top Q^i X_t^\theta \right) \, dt + \gamma^i \lvert X_T^{\theta,i} - d^i \rvert^2 \right],
\end{equation}
with $X^\theta =  (X^{\theta,i})_{i\in I}$. 
 Similarly, the potential function \eqref{eq:Phi_general_def} takes the form 
\begin{equation}
    \label{eq:potential_function_affine}
    \begin{split}
        \Phi(\theta) &= \mathbb{E} \Bigg[ \int_0^T \left( \sum_{i=1}^N |{K}^i_t (X_t^{\theta,i} - \mathbb{E}[X_t^{\theta, i}]) + G^i_t|^2 + (X_t^\theta)^\top Q X_t^\theta \right) \, dt   + \sum_{i=1}^N\gamma^i |X_T^{\theta, i} - d^i|^2 \Bigg].
    \end{split}
\end{equation}
Write $\mu^{G}=(\mu^{G,i} )_{i\in I}  $
and $\vartheta^{K}=( \vartheta^{K,i})_{i\in I}  $, where   for each $i\in I$, 
    $\mu_t^{G,i} = \mathbb{E}[X_t^{\theta, i}]$ and $\vartheta_t^{K,i} = \mathbb{V}[X_t^{\theta, i}]$ 
  satisfy:
\begin{equation}
    \label{eq:moments}
    \begin{split}
        &\frac{\partial \mu_t^{G,i}}{\partial t} = G_t^i, \quad t \in [0,T]; \quad \mu_0^{G,i} = \mathbb{E}[\xi^i], \\
        &\frac{\partial \vartheta_t^{K,i}}{\partial t} = 2 K_t^i \vartheta_t^{K,i} + (\sigma_t^i)^2, \quad t \in [0,T]; \quad \vartheta_0^{K,i} = \mathbb{V}[\xi^i].
    \end{split}
\end{equation}
The state mean $\mu^{G}   $ depends only on the intercept parameters  $G$, whereas the state variance $\vartheta^{K}$ depends only on the slope parameters  $K$.   
 Define the decomposed cost functionals  
$J^{1,i}: \mathcal{K} \to \mathbb{R}$ and $J^{2,i}: \mathcal{G} \to \mathbb{R}$ by 
\begin{equation}
    \label{eq:cost_functionals_decomp}
    \begin{split}
        &J^{1,i}(K) \coloneqq \int_0^T \left( (K_t^i)^2 \vartheta_t^{i} + \sum_{j = 1}^N Q_{j,j}^i \vartheta_t^{j} \right) \, dt + \gamma^i \vartheta_T^{i}, \\
        &J^{2,i}(G) \coloneqq \int_0^T \left( (G_t^i)^2 + \mu_t^\top Q^i \mu_t \right) \, dt + \gamma^i (\mu_T^{i} - d^i)^2,
    \end{split}
\end{equation}
where  we omit the dependence on $K$ and $G$ in the superscripts of $\mu^{G}$ and $\vartheta^{K}$ for notational  simplicity. 
Similarly, define the decomposed potential functions 
$\Phi^1: \mathcal{K} \to \mathbb{R}$ and $\Phi^2: \mathcal{G} \to \mathbb{R}$ by 
\begin{equation}
    \label{eq:potential_functions_decomp}    
    \begin{split}
        &\Phi^1(K) \coloneqq \sum_{i=1}^N \left[ \int_0^T   \left( (K_t^i)^2 + Q_{i,i} \right) \vartheta_t^{i} \, dt + \gamma^i \vartheta_T^{i} \right],\\
        &\Phi^2(G) \coloneqq \int_0^T \left( G_t^\top G_t + \mu_t^\top Q \mu_t \right)  \, dt + (\mu_T - d)^\top \Lambda (\mu_T - d).
    \end{split}
\end{equation}

The following proposition shows that $J^{1,i}$ and $J^{2,i}$ decompose the original cost $J^i$ and admit a potential structure associated with the corresponding decomposed potential functions $\Phi^1$ and $\Phi^2$. 
\begin{proposition}
    \label{prop:potential_decomposition}
    For all   $i \in I$ and   $(K,G)  \in \mathcal{K} \times \mathcal{G}$,
    \begin{equation*}
        J^i(K,G) = J^{1,i}(K) + J^{2,i}(G),\quad 
         \Phi(K, G) = \Phi^1(K) + \Phi^2(G).
    \end{equation*}
      Assume further that Assumption  \ref{assumption:Qij_symmetry}  holds. Then $\Phi^1$ is a potential function for the game $(I,(J^{1,i})_{i\in I}, \cK)$ such  that for all $i\in I$, $K^i,\tilde K^i\in \cK^i$ and $K^{-i} \in \cK^{-i}$,
    $$
 J^{1,i}( \tilde K^i, K^{-i} ) - J^{1,i}( K^i, K^{-i} )= 
      \Phi^1( \tilde K^i, K^{-i} ) - \Phi^{1}( K^i, K^{-i} ). 
    $$
    Moreover,  $\Phi^2$ is a potential function for the game $(I,(J^{2,i})_{i\in I}, \cG)$. 
\end{proposition}

By Proposition \ref{prop:potential_decomposition}, 
$   \nabla_{K^i} J^{i}( K, G ) =  \nabla_{K^i} \Phi^{1}( K )$ and $ 
         \nabla_{G^i} J^{i}( K,G ) =   \nabla_{G^i} \Phi^2( G )$.
  Moreover,
   the convergence analysis of Algorithm~\ref{algo} reduces to studying the landscapes of the mappings $K \mapsto \Phi^1(K)$ and $G \mapsto \Phi^2(G)$.

We now state two structural properties of the maps $\Phi^1$ and $\Phi^2$. Crucially, these properties hold without the symmetric interaction assumption
(Assumption \ref{assumption:Qij_symmetry}), which will be used to  analyze  learning dynamics in games with asymmetric interactions in Section~\ref{section:asymmetric}.
For the functional $\Phi^1$, we establish a gradient dominance property,
which quantifies the sub-optimality of any parameter $K$  using the normalized gradient 
$\cD^{\Phi}_K$ of $
\Phi^1$ at $K$,   defined by 
\begin{equation}
     \label{eq:DK}
     \cD^{\Phi}_K=(\cD_K^{\Phi,i})_{i\in I}\in L^2([0,T],\sR^N),
     \quad 
   \textnormal{with $ (\mathcal{D}^{\Phi,i}_K)_t \coloneqq  (\nabla_{K^i} \Phi^1(K))_t (\vartheta_t^{K,i})^{-1} $ }.
\end{equation}

\begin{proposition}
    
    \label{prop:graddom_K}
    Suppose Assumption \ref{assumption:Q_positivesemidef} holds,
    and let  $K^{\Phi,*}=(K^{\Phi,*,i})_{i\in I}$ be defined in 
    \eqref{eq:optimal_policy}. Then $K^{\Phi,*}$ is the unique minimizer of $\Phi^1:
    \cK\to \sR$, and for all $K \in \mathcal{K}$,
    \begin{equation*}
        0 \le  \Phi^1(K)-
        \Phi^1(K^{\Phi,*})  \le   \frac{M^{\vartheta,*}}{4} \lVert \mathcal{D}^\Phi_K \rVert_{L^2}^2,
        \quad \textnormal{with
        $M^{\vartheta,*} \coloneqq   \max_{i\in I }\|\vartheta^{K^{\Phi,*},i}\|_{L^\infty}>0.$}
    \end{equation*}

\end{proposition}

The minimizer of $\Phi^1$ coincides with the parameter $K^{\Phi,*}$ of the NE  given in Theorem~\ref{theorem:optimalpolicy}, under the additional Assumption~\ref{assumption:Qij_symmetry} of symmetric interaction.
For  general interaction weights, the minimizer 
$K^{\Phi,*}$ exists but may  not coincide with the NE policies.

For the functional $\Phi^2$, we prove the     strong convexity and Lipschitz smoothness over   $\cG$.
We denote by 
$\nabla_G \Phi^2(G)=(\nabla_{G^1}\Phi^2(G), \ldots, \nabla_{G^N}\Phi^2(G))^\top $
the derivative of the map $G\mapsto \Phi^2(G)$.

\begin{proposition}
    \label{prop:strongconvex_lipschitz_G}
    Suppose Assumption \ref{assumption:Q_positivesemidef} holds. 
    Define 
    $m \coloneqq 2$ and $L \coloneqq 2 + T^2 \lambda_{\mathrm{max}}(Q_{\mathrm{sym}}) + 2 T \max_{i \in I} \gamma^i$.
    Then for all $G, \Tilde{G} \in \mathcal{G}$,
    \begin{align*}
        \frac{m}{2} \lVert \Tilde{G} - G \rVert_{L^2}^2\le 
        \Phi^2(\Tilde{G}) - \Phi^2(G) - \langle \Tilde{G} - G, \nabla_G \Phi^2(G) \rangle_{L^2} \le   \frac{L}{2} \lVert \Tilde{G} - G \rVert_{L^2}^2.
    \end{align*}
\end{proposition}

Using Propositions \ref{prop:graddom_K} and \ref{prop:strongconvex_lipschitz_G}, the following theorem establishes  
the  global linear convergence of Algorithm \ref{algo} to the NE policy $ (K^{\Phi,*},G^{\Phi,*})$  given in Theorem \ref{theorem:optimalpolicy}.

\begin{theorem}
    \label{theorem:convergence_potentialcase}
    Suppose Assumptions \ref{assumption:Q_positivesemidef} and \ref{assumption:Qij_symmetry} hold. Let $m, L > 0$ be the constants   in Proposition \ref{prop:strongconvex_lipschitz_G}, and  $(K^{(0)}, G^{(0)}) \in L^\infty([0,T], \mathbb{R}^N) \times L^2([0,T], \mathbb{R}^N)$.
    There exist constants $C_1^K, C_2^K, C_3^K > 0$ such that if $\eta_K^i \in (0, C_1^K)$ and  $\eta_G^i \in (0, 1/L)$ for all $i \in I$,
      the iterates 
      $(K^{(\ell)}, G^{(\ell)})_{\ell \in \mathbb{N}_0}$ generated by 
      Algorithm \ref{algo} satisfy for all $\ell \in \sN_0$,
    \begin{align*}
        \lVert K^{(\ell)} - K^{\Phi,*} \rVert_{L^2}^2 & \le C_2^K  \left( \Phi^1(K^{(0)}) - \Phi^1(K^{\Phi,*}) \right) \left( 1 - \eta_K^{\mathrm{min}} C_3^K \right)^{\ell },
        \\
        \lVert G^{(\ell)} - G^{\Phi,*} \rVert_{L^2}^2 &\le \frac{\eta_G^{\mathrm{max}}}{\eta_G^{\mathrm{min}}} \lVert G^{(0)} - G^{\Phi,*} \rVert_{L^2}^2 \left( 1 - \eta_G^{\mathrm{min}} m \right)^{\ell} ,
    \end{align*}
    where
    $\eta_{K}^\mathrm{min} \coloneqq \min_{i \in I} \eta_K^i$,  
    $\eta_{G}^\mathrm{min} \coloneqq \min_{i \in I} \eta_G^i$ and $\eta_{G}^\mathrm{max} \coloneqq \max_{i \in I} \eta_G^i$.
\end{theorem}

The precise expressions of $C_1^K$, $C^K_2$ and $C^K_3$ are given in \eqref{eq:CK123}, which do not depend explicitly on   the number of players $N$.
Theorem~\ref{theorem:convergence_potentialcase} implies that the computational complexity of Algorithm~\ref{algo} to achieve an error of $\varepsilon$ scales \emph{linearly   with $N$ and logarithmically with $1/\varepsilon$.}  
 
\begin{corollary}
\label{cor:approximate_NE_symmetric}
Assume the same conditions as in Theorem~\ref{theorem:convergence_potentialcase}. 
For all $\varepsilon > 0$, there exists $M \in \mathbb{N}_0$, depending linearly on $\log{(1/\varepsilon)}$,
 such that 
for all $\ell \ge M$, 
the  policy profile $ u_{\theta^{(\ell)}}$, with $\theta^{(\ell)} = (K^{(\ell)}, G^{(\ell)})$ generated by Algorithm~\ref{algo}, 
is an $\varepsilon$-NE for the game $\sG = (I, (J^i)_{i \in I},  {\cV})$.
\end{corollary}

\section{Learning approximate NEs under asymmetric interactions}
\label{section:asymmetric}

This section considers the LQ game $\sG$ without  Assumption \ref{assumption:Qij_symmetry}. Although the function $\Phi$ in \eqref{eq:Phi_general_def} is no longer  an exact potential function for the game $\sG$, we can still leverage it to design convergent independent learning algorithms for   approximate NEs of $\sG$.

\subsection{Characterization of  distributed NEs}

This section presents a verification theorem for affine NEs of the game $\sG$ with general interactions. This  characterization is   of independent theoretical interest, 
and will be used in Section \ref{section:numerical} to construct benchmarks for numerical experiments.

Consider the following coupled   ODE system: 
for all $i \in I$ and $t \in [0,T]$,
\begin{subequations}
    \label{eq:ODEs_asymmetric}
    \begin{align}[left = \empheqlbrace\,]
        &\frac{\partial P_t^{i}}{\partial t} - (P_t^{i})^2 + Q_{i,i}^i = 0, \quad P_T^{i} = \gamma^i,
        \label{eq:P_asymmetric}\\
        &\frac{\partial \lambda_t^i}{\partial t} - P_t^i \lambda_t^i + \sum_{j \in I\setminus \{i\}} Q_{i,j}^i \mu_t^j = 0, \quad \lambda_T^i = - \gamma^i d^i, 
        \label{eq:lambda_asymmetric}\\
        &\frac{\partial \mu_t^i}{\partial t} = - P_t^i \mu_t^i - \lambda_t^i, \quad \mu_0^i = \mathbb{E}[\xi^i] \label{eq:ODE_system_asym_mean}.
    \end{align}
\end{subequations}

Note that, in contrast to the ODE system \eqref{eq:ODEs_valuefunction} for the symmetric case, the system \eqref{eq:ODEs_asymmetric} is coupled across all players through \eqref{eq:lambda_asymmetric}.
The following theorem constructs NEs for the game $\sG$ 
through   solutions to \eqref{eq:ODEs_asymmetric}.

\begin{theorem}
    \label{theorem:optimalpolicy_asymmetric}
    Suppose that  \eqref{eq:ODEs_asymmetric} has a   solution $(P^i)_{i\in I}, (\mu^i)_{i\in I}, (\lambda^i)_{i\in I} \in \mathcal{C}([0,T], \mathbb{R}^N)$. Define $u^{*} \in \mathcal{V}$ such that 
    for all $i\in I$ and $(t,x) \in [0,T] \times \mathbb{R}$, 
    \begin{equation}
        \label{eq:NE_policies_def}
        u^{*,i}(t,x) \coloneqq 
        - P_t^{i} x - \lambda_t^i.
    \end{equation}
    Then $u^* $ 
    is an NE of the game $\sG = (I, (J^i)_{i \in I},  {\cV})$. 

\end{theorem}

Theorem \ref{theorem:optimalpolicy_asymmetric} 
generalizes Theorem \ref{theorem:optimalpolicy}
to arbitrary interaction weights $(Q^i)_{i\in I}$.
The policy 
\eqref{eq:NE_policies_def} can be written in the form 
\eqref{eq:optimal_policy} 
as $  u^{*,i}(t,x) \coloneqq K_t^{*,i} (x - \sE[X^{u^{*,i},i}_t]) + G_t^{*,i}$, where 
$  
        K_t^{*,i} \coloneqq - P_t^{i}$ and $ G_t^{*,i} \coloneqq - \lambda_t^i - P_t^{i} \mu_t^i$.
Under    Assumptions \ref{assumption:Q_positivesemidef}  and \ref{assumption:Qij_symmetry}, these policy coefficients  
coincide with those of the NE   \eqref{eq:optimal_policy}
in Theorem \ref{theorem:optimalpolicy}.

For completeness, we provide an analogous condition to Assumption \ref{assumption:Q_positivesemidef} that ensures the well-posedness of \eqref{eq:ODEs_asymmetric}, and hence the existence of an NE  for the game $\sG$ with asymmetric interactions. Note that this condition will not be imposed in the subsequent analysis of learning algorithms in Section \ref{subsection:convergence_asymmetric}, since our focus there is on approximate NEs.

\begin{proposition}
    \label{prop:ODE_system_asym_sol}
    Define   $\Hat{Q} =(\Hat{Q}_{i,j})_{i,j\in I} \in \mathbb{R}^{N \times N}$  by $\Hat{Q}_{i,j} \coloneqq Q_{i,j}^i$ for all $i,j \in I$.
    If the following boundary value problem admits only the trivial solution 
    $y \equiv \mathbf{0}$:
    \begin{equation}
        \label{eq:BVP_mu_auxiliary}
        \frac{\partial^2 y_t}{\partial t^2}  = \Hat{Q} y_t, \quad t \in [0,T]; \quad y_0 = \mathbf{0}, \quad \frac{\partial y_T}{\partial t}  = - \Lambda y_T,
    \end{equation}
    then   the  system \eqref{eq:ODEs_asymmetric} has a unique solution, and 
    $u^*$ defined in   \eqref{eq:NE_policies_def}
    is an NE for the game $\sG = (I, (J^i)_{i \in I},  {\cV})$.
    In particular, this condition holds if  
    $\Hat{Q}_{\mathrm{sym}} \coloneqq (\Hat{Q} + \Hat{Q}^\top)/2 \in \sS^N_{\ge 0}$. 
 
\end{proposition}

\subsection{Independent  policy gradient algorithm and its convergence}
\label{subsection:convergence_asymmetric}

This section proposes  an independent  policy gradient algorithm for the game $\sG$ with asymmetric interactions,
and analyzes its convergence through   the $\alpha$-potential function $\Phi$ defined in \eqref{eq:Phi_general_def}.

\subsubsection{Projected policy gradient algorithm}

To approximate NEs with asymmetric interactions, we consider affine policies as defined in \eqref{eq:policy_parameterization}, but require the policy parameters to satisfy certain   a priori bounds.
This allows us to control the misalignment between individual policy gradients and the gradient of the $\alpha$-potential function  $\Phi$.

Specifically, 
we    modify Algorithm \ref{algo} by replacing the gradient descent update for $G$ in \eqref{eq:potential_GD_update} with a projected gradient descent update. 
For each iteration $\ell\in \sN_0$, given the parameter profile 
$\theta^{(\ell)}\coloneqq (K^{(\ell)}, G^{(\ell)})$,
consider the following update: for all $i\in I$, 
\begin{equation}
    \label{eq:alphapot_G_update_projection}
    \begin{split}
        &K_t^{(\ell+1),i} \coloneqq K_t^{(\ell),i} - \eta_K^i (\nabla_{K^i} J^i(\theta^{(\ell)} ) )_t (\vartheta_t^{(\ell),i})^{-1}, \quad t\in [0,T] \\
        &{G_t^{(\ell+1),i} \coloneqq  \cP_{\overline{C}_G} \big(G^{(\ell),i} - \eta_G^i \nabla_{G^i} J^i(\theta^{(\ell)} )\big)_t}, \quad t\in [0,T],
    \end{split}
\end{equation}
where $\eta_K^i,\eta^i_G>0$
are given learning rates, 
$\nabla_{K^i} J^i$ and $\nabla_{G^i} J^i$ are the  derivatives of $J^i$ with respect to $K^i$ and $G^i$, respectively, 
 $\overline{C}_G>0$ is a given constant, 
and $\cP_{\overline{C}_G}$
is the (orthogonal) projection   
onto the $L^2$-ball with radius $\overline{C}_G$: 
$$
\cP_{\overline{C}_G}(f) =\min\left(1,\frac{\overline{C}_G }{\|f\|_{L^2}}\right)f, \quad \forall f\in L^2([0,T],\sR).
$$

\begin{remark}[\textbf{Implicit Regularization in $K$}]
    In \eqref{eq:alphapot_G_update_projection}, we only project the intercept parameter $G$, while keeping the normalized gradient descent update for $K$ as in \eqref{eq:potential_GD_update}. This is because the normalized gradient descent automatically ensures that the policy iterates $(K^{(\ell)})_{\ell \in \mathbb{N}}$ are uniformly bounded, even in the setting with asymmetric interactions (see Proposition \ref{prop:uniform_bound_K}). This {implicit regularization property} is crucial for the convergence analysis of policy gradient methods in continuous-time problems 
    \cite[Remark 2.3]{giegrich2024convergence}.
\end{remark}

Below we summarize the 
 policy gradient algorithm for games    with asymmetric interactions.  

\begin{algorithm}[H]
\caption{Independent Policy Gradient Learning for Asymmetric Interactions}
\label{algo_projected}
\begin{algorithmic}[1]
\STATE \textbf{Input:} 
Initial parameters $(K^{(0),i}, G^{(0),i})_{i\in I}$,
learning rates $ \eta_K^i, \eta_G^i\in (0,\infty)$ for all $i\in I$, 
and the projection threshold $\overline{C}_G>0$. 
\FOR{$\ell = 1, 2, \dots$}
    \STATE 
    Obtain the updated parameters  $ (K^{(\ell),i},G^{(\ell),i} )_{i\in I}$ by \eqref{eq:alphapot_G_update_projection}. 
\ENDFOR
\end{algorithmic}
\end{algorithm}

 \subsubsection{Convergence analysis}

Employing the affine policies in \eqref{eq:policy_parameterization} implies that both the individual cost $J^i$ 
and the $\alpha$-potential 
function $\Phi$
remain  decomposable as shown in Proposition \ref{prop:potential_decomposition}, namely,
 $J^i(K,G) = J^{1,i}(K) + J^{2,i}(G)$ and $ \Phi(K, G) = \Phi^1(K) + \Phi^2(G)
$, 
where 
$J^{1,i}$ and $J^{2,i}$  
are 
defined in \eqref{eq:cost_functionals_decomp},
and  
  $\Phi^{1}$ and $\Phi^{2}$ are
defined in \eqref{eq:potential_functions_decomp}.
This allows us to analyze the updates for $K$ and $G$ in Algorithm~\ref{algo_projected} separately. 
However, unlike the symmetric-interaction setting of Section~\ref{section:symmetric}, the game $\sG$ is no longer an exact potential game, and the gradients   $\nabla_{K^i}J^{i}$ and $\nabla_{K^i}\Phi^1$ (resp.~$\nabla_{G^i}J^{i}$ and $\nabla_{G^i}\Phi^2$) differ.
To analyze Algorithm~\ref{algo_projected}, we interpret it as a biased gradient descent on the functions $\Phi^1$ and $\Phi^2$, and   quantify the resulting bias in terms of the constant $C_Q$.

The following lemma analyzes  the gap between the gradients of $J^i$ and $\Phi^1$ with respect to $K$.
The result leverages the uniform bound $\sup_{\ell \in \mathbb{N}_0} \|K^{(\ell)}\|_{L^\infty} \le \overline{C}_\infty^K$ (see Proposition~\ref{prop:uniform_bound_K}) and the Lipschitz stability of the ODE characterizations of $\nabla_{K^i} J^i$ and $\nabla_{K^i} \Phi^1$.

\begin{lemma}
    \label{lemma:K_gradient_bound}
     Let $K^{(0)} \in L^\infty([0,T], \mathbb{R}^N)$, and  
       $\eta^i_K\in (0,1/2)$ for all   $i\in I$.
     There exists $\overline{C}_\infty^K>0$ such that 
     the iterates  $ (K^{(\ell)},G^{(\ell)})_{\ell \in \mathbb{N}}$  from Algorithm~\ref{algo_projected} satisfy  for all $\ell\in \sN$,
     $$
     \max_{i \in I} \| \nabla_{K^i} J^i(K^{(\ell)},G^{(\ell)} )  (\vartheta^{K^{(\ell)},i})^{-1}- \nabla_{K^i} \Phi^1(K^{(\ell)})  (\vartheta^{K^{(\ell)},i})^{-1}\|_{L^2}
     \le \frac{\exp{\left( 2 \overline{C}_\infty^K T \right)}}{ 4 (\overline{C}_\infty^K)^{3/2}}C_Q.
     $$
    
\end{lemma}

To quantify the 
  gap between  $\nabla_{G^i}J^i$ and $\nabla_{G^i}\Phi^2$, we introduce the following   bounds  on the moments of the  state processes:
  for each $\mathfrak B>0$, let 
  \begin{equation}
  \label{eq:moment_bound_afffine}M^\mu_{\mathrm{aff}}(\mathfrak B) \coloneqq \max_{i \in I} \sup_{G \in \cG^i, \|G\|_{L^2}\le \mathfrak B  }\lVert \mu^{G,i}  \rVert_{L^2},
    \quad M^\vartheta_{\mathrm{aff}}(\mathfrak B) \coloneqq \max_{i \in I} \sup_{K \in \cK^i, \|K\|_{L^\infty}\le \mathfrak B  }\lVert \vartheta^{K,i} \rVert_{L^1},
  \end{equation}
where  
$\mu^{G,i} $ and $\vartheta^{K,i}$ are   defined by \eqref{eq:moments}.
Note that 
for the state process \eqref{eq:state_dynamics} controlled by an affine policy
 in $ \mathcal V_{\mathrm{aff}}^i  $,
the mean $\mu^{G,i}$ 
and the variance  
$\vartheta^{K,i}$
  depend  only on the intercept parameter $G^i$ and the slope parameter $K^i$ of the policy, respectively.

\begin{lemma} 
    \label{lemma:gradient_bound}
     For all $\mathfrak B>0$ and  $(K,G) \in \mathcal{K} \times \mathcal{G}$ with $\max_{i\in I}\|G^i\|_{L^2}\le \mathfrak B$,
    \begin{equation*}
      \lVert \nabla_{G} J(K,G) - \nabla_{G} \Phi^2(G)  \rVert_{L^2} \le 
      2 T \sqrt{N} C_Q M^\mu_{\mathrm{aff}}(\mathfrak B) ,
    \end{equation*}
    where $\nabla_G J(K,G) \coloneqq  (\nabla_{G^1} J^1(K,G), \dots, \nabla_{G^N} J^N(K,G))^\top$.
\end{lemma}

Using Lemmas 
\ref{lemma:K_gradient_bound}
and \ref{lemma:gradient_bound},
together with the structural properties of 
$\Phi^1$ and $\Phi^2$ established in Section~\ref{section:symmetric}, we show that
  the iterates  $ (K^{(\ell)},G^{(\ell)})_{\ell \in \mathbb{N}}$  produced by Algorithm~\ref{algo_projected} yield approximately optimal polices for the function $\Phi$ (Propositions \ref{proposition:convergence_alphapotential_slope} and \ref{proposition:convergence_alpha_potential_projection}). 
To show they are approximate NEs for the game, we introduce  the following policy class of player $i$: 
\begin{equation}
    \label{eq:policy_bdd_moments}
    \mathcal{V}_{\mathrm{b}}^i  \coloneqq 
    \left\{ u \in \cV^i 
    \,\middle\vert\, 
    \begin{aligned}
        &\lVert \mathbb{E}[X^{u,i}] \rVert_{L^2} \le M_b^\mu, \lVert \mathbb{V}[X^{u,i}] \rVert_{L^1} \le M_b^\vartheta, \\ 
        &\textnormal{$X^{u,i}$ satisfies \eqref{eq:state_dynamics} controlled by $u$}  
    \end{aligned} \right\}
\end{equation}
for some given $M_{\mathrm{b}}^\vartheta, M_{\mathrm{b}}^\mu > 0$,  
and define $\mathcal{V}_{\mathrm{b}}  \coloneqq \prod_{i\in I}\mathcal{V}_{\mathrm{b}}^i$.
 Standard moment estimates of \eqref{eq:state_dynamics} show that 
$\mathcal{V}_{\mathrm{b}} \subset \cV$ includes all  nonlinear  policies exhibiting appropriate linear growth.
To facilitate the analysis, 
we assume 
that the policy class $\mathcal V_{\mathrm{b}}$
is sufficiently large  as specified below.

\begin{assumption}
    \label{assumption:policy_class}
    The projection threshold  
 $\overline{C}_G$ in \eqref{eq:alphapot_G_update_projection}
 satisfies 
    $\overline{C}_G\ge \max_{i\in I}\lVert G^{\Phi,*,i} \rVert_{L^2}$, with
    $G^{\Phi,*}$   defined in \eqref{eq:optimal_policy}.
    The constants 
    $M^\mu_{\rm b}$ and $M^\vartheta_{\rm b}$ in \eqref{eq:policy_bdd_moments}
    satisfy 
    $M_{\mathrm{b}}^\mu \ge M^\mu_{\mathrm{aff}}(\overline{C}_G) $    and 
    $M_{\mathrm{b}}^\vartheta\ge M^\vartheta_{\mathrm{aff}}(\overline{C}_\infty^K )$, 
    where $\overline{C}_\infty^K>0$ is defined in \eqref{eq:C0K} and depends on
    $K^{(0)}$ in Algorithm \ref{algo_projected}.

\end{assumption}

Assumption \ref{assumption:policy_class}
ensures that the 
  policies generated by Algorithm~\ref{algo_projected} belong to  $\mathcal V_{\mathrm{b}}$.
  Indeed, for 
  all $\ell\in \sN$ and $i\in I$,
$\|K^{(\ell),i}\|_{L^\infty} \le \overline{C}_\infty^K$ due to Proposition~\ref{prop:uniform_bound_K},
   and  $ \|G^{(\ell),i}\|_{L^2} \le \overline{C}_G$ due to the explicit projection.  Hence, setting $\theta^{(\ell),i}=(K^{(\ell),i},G^{(\ell),i})$,  the affine structure of the policy $u_{\theta^{(\ell),i}}\in   \mathcal V_{\mathrm{aff}}^i $
   implies that the corresponding state satisfies the moment bounds in  \eqref{eq:policy_bdd_moments}.

The following theorem is analogous to Corollary~\ref{cor:approximate_NE_symmetric} and shows that Algorithm~\ref{algo_projected} yields approximate NEs for the game $\sG$ with asymmetric interactions. 

\begin{theorem}
\label{theorem:convergence_asymmetric_all}
    Suppose Assumptions \ref{assumption:Q_positivesemidef} and \ref{assumption:policy_class} hold. 
    Let $(K^{(0)}, G^{(0)}) \in L^\infty([0,T], \mathbb{R}^N) \times L^2([0,T], \mathbb{R}^N)$ with $\max_{i\in I}\|G^{(0),i}\|_{L^2}\le \overline {C}_G$. 
    Then 
    there exists $\bar{\eta} > 0$ 
    such that 
    if the learning rates satisfy 
      $\eta_K^i, \eta_G^i \in ( 0,\bar{\eta})$ for all $i \in I$,
      and $\eta_G^{\mathrm{min}} > \eta_G^{\mathrm{max}} / (1 + 2 \eta_G^{\mathrm{max}})$, the following holds: 
      for all $\varepsilon > 0$, there exists $M \in \mathbb{N}_0$, depending linearly on $\log{(1/\varepsilon)}$,  such that 
for all $\ell \ge M$,
the  policy profile $ u_{\theta^{(\ell)}}$, with $\theta^{(\ell)} = (K^{(\ell)}, G^{(\ell)})$ generated by Algorithm~\ref{algo_projected}, satisfies 
  \begin{equation*}
        J^i(u_{\theta^{(\ell)}}) \le  J^i(u^i, u_{\theta^{(\ell)}}^{-i}) +( \varepsilon + \delta(C_Q)),
        \quad \forall u^i \in \mathcal{V}_{\mathrm{b}}^i, \quad i \in I,
    \end{equation*}
 where the constant $\delta(C_Q) > 0$ is 
 defined in \eqref{eq:delta}.
That is, 
the  policy profile $ u_{\theta^{(\ell)}}$
is an $(\varepsilon + \delta(C_Q))$-NE 
 for the game $\sG = (I, (J^i)_{i \in I},  {\cV}_{\mathrm{b}})$.
  
\end{theorem}

Theorem \ref{theorem:convergence_asymmetric_all}  provides a non-asymptotic performance guarantee for Algorithm~\ref{algo_projected}, without requiring the limit $N \to \infty$ or imposing an asymptotic behavior of the interaction matrices   $(Q^i)_{i \in I}$.
The constant $\delta(C_Q)$
is of   order  $N  C_Q^2  $,
which implies that  
Algorithm \ref{algo_projected} attains an approximate NE with   complexity scaling linearly in $N$, 
provided that 
as $N\to \infty$,
$\lvert Q_{i,j}^i - Q_{j,i}^j \rvert =o(1/N^r)$ 
for some $r> 1.5$.
This 
contrasts with (graphon) mean field game approximations of large-population games, which   require
all  interaction weights  $Q^i_{k,j}$, $i,k,j\in I$,  to vanish as $N \to \infty$.

\section{Numerical experiments}
\label{section:numerical}

This section evaluates the performance of Algorithm~\ref{algo} using the cost functional in Example~\ref{example:nr1}. Our experiments confirm that Algorithm~\ref{algo} exhibits robust linear convergence to the NE across both symmetric and asymmetric interaction networks.

\paragraph{Experiment setup.}
Consider the $N$-player LQ games $\sG$ 
with state \eqref{eq:state_dynamics} and cost  functional \eqref{eq:cost_crowd_motion_lq}.
We set $N=10$, $T=1$, and 
$d = (-4,-3,\cdots,4,5)^\top$.
For all $i\in I$,
we take 
$\sigma^i=0.25$, 
 $\gamma^i=1$, 
 and 
$\xi^i\sim \mathcal N( \mu^i_0, 0.01)$, 
where $\mu_0 = (5,4,\dots,-3,-4)^\top$.

  We implement Algorithm \ref{algo} by discretizing the state dynamics, cost functional, and policy class as in \cite{giegrich2024convergence, plank2025policy}. 
Consider $N_{\mathrm{t}}\in\mathbb{N}$
and 
a uniform time mesh $(t_j)_{j=0}^{N_{\mathrm{t}}} \subset [0,T]$ with $t_j = j\Delta t$,  where   $\Delta t= T/N_{\mathrm{t}}$.
We define piecewise-constant slope and intercept parameters 
$( K_{t_j}^i, G_{t_j}^i)_{0\le j\le  N_{\mathrm{t}}-1, i\in I }$
on this grid, which will be updated   using Algorithm \ref{algo}.
Given a piecewise-constant policy, 
we approximate 
the associated cost functional \eqref{eq:cost_crowd_motion_lq}  by
\begin{equation}
    \label{eq:cost_approx}
    \begin{split}
        \Hat{J}^i(K,G) \coloneqq \frac{1}{N_{\mathrm{s}}} &\sum_{k = 1}^{N_{\mathrm{s}}} \Bigg[ \sum_{j = 0}^{N_{\mathrm{t}} - 1} \bigg( \left( K_{t_j}^i (X_{t_j}^{i, (k)} - \Hat{\mu}_{t_j}^i) + G_{t_j}^i \right)^2 \\
        &\quad+ \sum_{\ell \in I\setminus \{i\}} \omega_{i,\ell} \left( X_{t_j}^{i, (k)} - X_{t_j}^{\ell, (k)} \right)^2 \bigg)\Delta t
        + \gamma^i \left(X_{T}^{i,(k)} - d^i \right)^2 \Bigg],
    \end{split}
\end{equation}
where 
for all $i\in I$,
$(X^{i,(k)}_{t_j})_{j=0}^{N_{\mathrm{t}}}$,
$k=1,\ldots, N_{\rm s}$,
are  sample  trajectories 
generated using an Euler-Maruyama discretization of
\eqref{eq:state_dynamics} on the time grid,  and $\hat{\mu}^i_{t_j}=\frac{1}{N_{\rm s}}\sum_{k=1}^{N_{\rm s}}X^{i,(k)}_{t_j}$.
 We set $N_{\mathrm{t}} = 200$ and $N_{\mathrm{s}} = 20{,}000$.
The interaction network $(\omega_{i,\ell})_{i,\ell\in I}$
will be specified below.

 We initialize policy  parameters  as $(K^{(0),i}, G^{(0),i}) \equiv 0$ for all $i \in I$. At each iteration $\ell\in \sN$, given the current parameters $(K^{(\ell),i}, G^{(\ell),i})_{i \in I}$, we update them by applying the gradient descent rule \eqref{eq:potential_GD_update} to the piecewise-constant policies:
  for all $i\in I$ and  $j=0,\ldots, N_{\rm t}-1$,
\begin{equation*}
    K_{t_j}^{(\ell + 1),i} = K_{t_j}^{(\ell),i} - \frac{\eta_K}{\Delta t \, \Hat{\vartheta}_{t_{j}}^{(\ell),i}} \widehat{\nabla_{K_{t_j}^i}} J^i(K^{(\ell)}, G^{(\ell)}), \quad   G_{t_j}^{(\ell + 1),i} = G_{t_j}^{(\ell),i} - \frac{\eta_G}{\Delta t} \widehat{\nabla_{G_{t_j}^i}} J^i(K^{(\ell)}, G^{(\ell)}),
\end{equation*}
where 
$\Hat{\vartheta}^{(\ell),i}_{t_j}$
is the empirical variance computed from the sample trajectories using the current parameters,
and 
 the gradients are
 obtained by applying automatic differentiation to 
\eqref{eq:cost_approx} in PyTorch.
 We  set the learning rates   
$\eta_K = \eta_G = 0.1$ and the number of iteration  to  $N_{\rm itr}=40$.

To quantify the convergence of
Algorithm \ref{algo},
we introduce the relative root mean squared error (RRMSE) of the policy parameter iterates with respect to the NE policy provided in   Theorem \ref{theorem:optimalpolicy_asymmetric}.
Specifically,
the RRMSE of the slope parameters
$K=(K^i)_{i\in I}$
is defined by
\begin{equation*}
  \mathrm{RRMSE}(K) \coloneqq \left(\sum_{i = 1}^{N} \sum_{j = 0}^{N_{\rm t}-1} (K_{t_j}^i - K_{t_j}^{i,*})^2 \right)^{1/2} \Bigg/\left(\sum_{i = 1}^{N} \sum_{j = 0}^{N_{\rm t}-1} (K_{t_j}^{i,*})^2 \right)^{1/2},
\end{equation*}
where the   
  reference  parameter $K^*$ is computed by solving the ODE system \eqref{eq:ODEs_asymmetric} using an explicit Runge-Kutta method on the time grid.  
The RRMSE of the intercept   $G$ is defined analogously.

\paragraph{Convergence with various networks.}

We first test Algorithm~\ref{algo} for symmetric interaction networks.
We run Algorithm~\ref{algo} over 10 independent trials, and in each trial construct   an  interaction graph using the randomly grown  uniform attachment model as in \cite[Example~11.39]{lovasz2012large}. Starting from a single node, at each iteration $n \leq N$ a new node is added, and every pair of previously non-adjacent nodes is connected independently with probability $1/n$. In particular, for all $i < j \leq n$, if $\omega_{i,j} = 0$ at the beginning of iteration $n$, then we set $\omega_{i,j} = 1$ with probability $1/n$ and $\omega_{i,j} = 0$ otherwise. Finally, we symmetrize the network by setting $\omega_{i,j} = \omega_{j,i}$. 

 Figure \ref{fig:num_experim_symmetric} exhibits the decay of the RRMSEs of policy parameters  with respect to the number of iterations,
where the solid line and the shaded area indicate the sample mean and the spread over 10 repeated experiments. 
Both policy parameters converge linearly to the NE, in agreement with Theorem~\ref{theorem:convergence_potentialcase}.
The seemingly larger noise at later iterations is due to the small magnitude of the errors, which makes fluctuations appear more pronounced on the logarithmic scale. The final error plateaus at a satisfactory level, influenced by the stochasticity in gradient estimation and the empirical approximation of the moments.

\begin{figure}[!ht]
    \centering

    \includegraphics[scale=0.48]{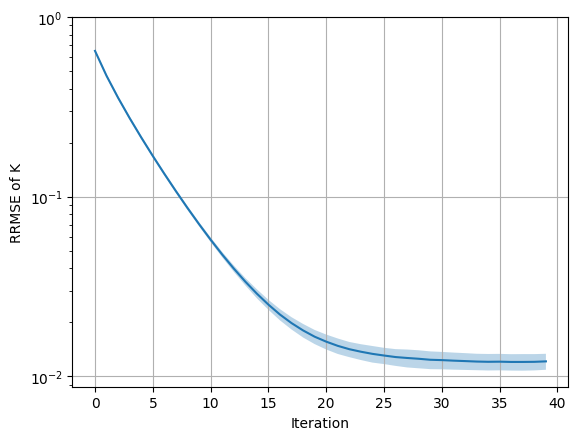}
    \quad 
    \includegraphics[scale=0.48]{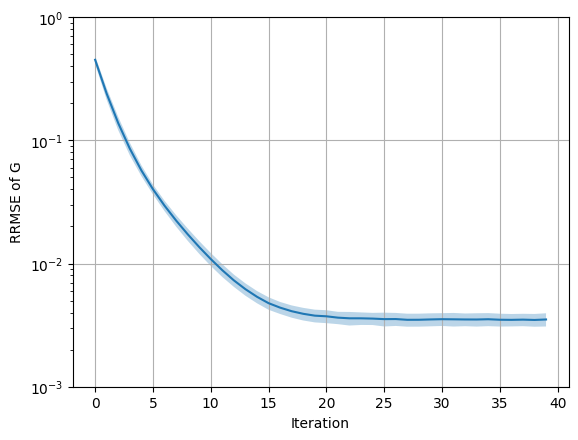}
   
    \caption{Convergence of Algorithm~\ref{algo} on symmetric uniform attachment networks. }
    \label{fig:num_experim_symmetric}
\end{figure}

We then test Algorithm~\ref{algo} on asymmetric interaction networks.
We perform 10 independent runs of Algorithm~\ref{algo}, and in each run generate a randomly sampled asymmetric Erd\H{o}s--R\'enyi interaction network: for all $i, j \in I$ with $i \neq j$, we sample $\omega_{i,j}$ from a Bernoulli distribution with parameter $p \in (0,1)$. That is, each directed edge $(i,j)$ is present independently with probability $p$.
Figure~\ref{fig:graphs} presents representative sampled networks for $p \in \{0.1, 0.5,  0.9\}$.

\begin{figure}[!ht]
    \centering

   \includegraphics[scale=0.33]{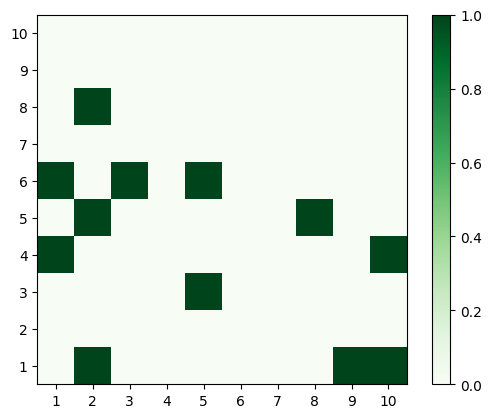}
   \quad 
   \includegraphics[scale=0.33]{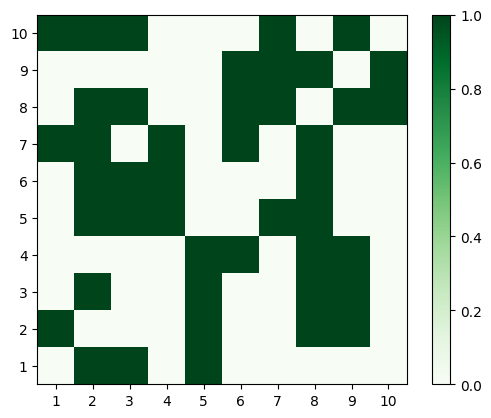}
   \quad
   \includegraphics[scale=0.33]{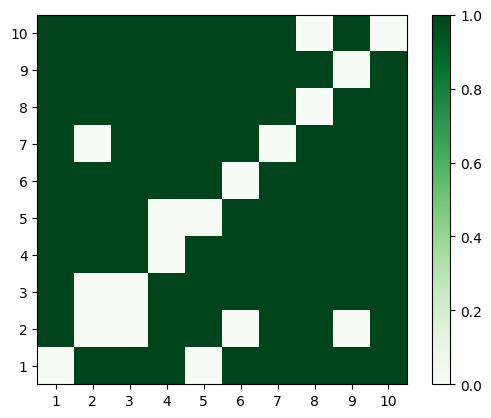}
  
    \caption{Asymmetric Erd\H{o}s--R\'enyi interaction networks with connection probability $p = 0.1$ (left), $p = 0.5$ (middle), and $p = 0.9$ (right).}
    \label{fig:graphs}
\end{figure}  

Figure~\ref{fig:num_experim_asymmetric} illustrates the performance of Algorithm~\ref{algo} on asymmetric Erd\H{o}s--R\'enyi networks with different connection probabilities. The results show that both parameters converge linearly to the NE policies until the error is dominated by the variance from Monte Carlo approximations of the policy gradients. The final accuracy remains stable across connection probabilities and is comparable to that for symmetric interaction networks in Figure~\ref{fig:num_experim_symmetric}, highlighting the robustness of Algorithm~\ref{algo}. 
This suggests that Theorem~\ref{theorem:convergence_asymmetric_all} provides a conservative bound on the algorithm  performance for games with  asymmetric interactions. 
Moreover,
it is evident from Figure~\ref{fig:num_experim_asymmetric} that the network structure has a greater impact on the convergence of the intercept parameter than on the slope parameter,
since the interaction matrices directly influence  both the policy gradient and the NE policy for 
$G$ (see Lemma \ref{lemma:cost_gradients} and Theorem \ref{theorem:optimalpolicy_asymmetric}).
Incorporating  additional structural properties of the interaction networks, such as sparsity, to better capture the algorithm’s behavior is an interesting direction for future work.

\begin{figure}[!ht]
    \centering
    
    \includegraphics[scale=0.48]{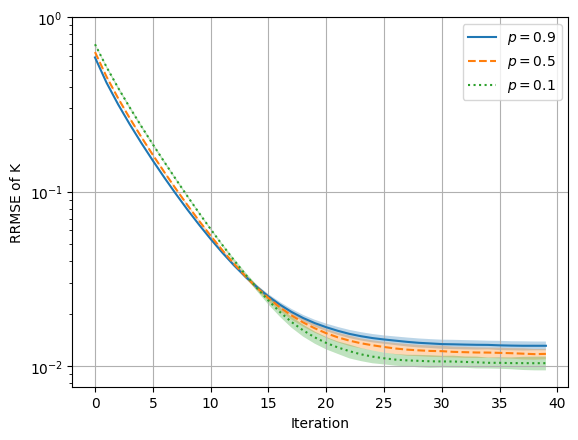}
    \quad 
    \includegraphics[scale=0.48]{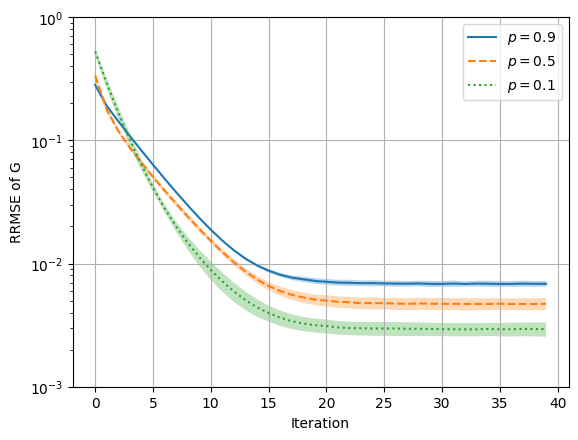} 
    \caption{Convergence of Algorithm \ref{algo}
    on asymmetric Erd\H{o}s--R\'enyi networks.
    }
    \label{fig:num_experim_asymmetric}
\end{figure}

\section{Proofs of Section \ref{section:distributed_equi}}

\begin{proof}[Proof of Proposition \ref{prop:Phi_is_alpha_potential}]
    For any $i \in I$, let $u^i,\tilde u^i\in \tilde{\cV}^i$ and $u^{-i}\in \tilde{\cV}^{-i}$. By the independence of the state processes, the cost and potential functions can be rewritten as 
    \begin{equation*}
        \begin{split}
            &J^i(u) = \int_0^T \bigg( \mathbb{E}[ \lvert u^i(t,X^i_t) \rvert^2] + \sum_{j = 1}^N \left( \mathbb{V}[X_t^j] Q_{j,j}^i \right) + (\mathbb{E}[X_t])^\top Q^i \mathbb{E}[X_t] \bigg) \, dt + \gamma^i \mathbb{E} [ \lvert X_T^i - d^i \rvert^2 ], \\
            &\Phi( u ) = \int_0^T \left( \sum_{j=1}^N \left( \mathbb{E}[\lvert u^j(t,X^j_t) \rvert^2] + \mathbb{V}[X_t^j] Q_{j,j} \right) + (\mathbb{E}[X_t])^\top Q \mathbb{E}[X_t] \right) \, dt + \sum_{i=1}^N\gamma^i \mathbb{E} [ \lvert X_T^i - d^i \rvert^2 ].
        \end{split}
    \end{equation*}
    This along with 
      the definition \eqref{eq:Q_matrix} of $Q$    shows that 
    \begin{equation*}
        \begin{split}
            &\lvert [\Phi(\Tilde{u}^i, u^{-i} ) - \Phi( u^i, u^{-i} ) ] - [J^i( \Tilde{u}^i, u^{-i} ) - J^i( u^i, u^{-i} )] \rvert \\
            &= \bigg\lvert \frac{1}{2} \int_0^T \Bigg( \left( \left( \mathbb{V}[X_t^i] - \mathbb{V}[\Tilde{X}_t^i] \right) + \left( (\mathbb{E}[X_t^i])^2 - (\mathbb{E}[\Tilde{X}_t^i])^2 \right) \right) \sum_{j \in I\setminus \{i\}} (Q_{i,j}^i - Q_{j,i}^j) \\
            &\qquad+ 2\left( \mathbb{E}[X_t^i] - \mathbb{E}[\Tilde{X}_t^i] \right) \sum_{j \in I\setminus \{i\}} \left( (Q_{j,i}^j - Q_{i,j}^i) \mathbb{E}[X_t^j] \right) \Bigg) \, dt \bigg \rvert \\
            &\le   \left(\max_{i \in I} \sup_{u^i \in \tilde{\cV}^i}  \|{\mathbb{V}[X^{u^i,i}]}\|_{L^1} + 3 \max_{i \in I} \sup_{u^i \in \tilde{\cV}^i}  {\|\mathbb{E}[ X^{u^i,i}]\|^2_{L^2}} \right) C_Q.
        \end{split}
    \end{equation*}
This completes the proof. 
\end{proof}

\section{Proofs of Section \ref{section:symmetric}}

\subsection{Proof of Theorem \ref{theorem:optimalpolicy}}

In the sequel, 
we denote by   $\mathcal{P}_2(E)$   the space of probability measures on
an Euclidean space 
$E$ with finite second moments. 
For each $\rho \in \mathcal{P}_2(\mathbb{R})$,
we write $\langle \rho, f(x) \rangle \coloneqq \int_{\mathbb{R}} f(x) \, \rho(d x)$ for the  integration of $f:\sR\to \sR$ with respect to $\rho$,
and for each  $\rho=\rho_1\otimes \cdots \otimes \rho_N \in \mathcal{P}_2(\mathbb{R})^N$, we write $\langle \rho, f (x)\rangle = (\langle \rho_i, f_i(x_i) \rangle )_{i=1}^N$ for the component-wise integration of $f:\sR^N\to \sR^N$. We denote by $\mathcal{D}_\rho$   the Lions derivative 
with respect to probability measures 
(see \cite[Chapter 5.2]{carmona_probabilistic_2018_I}). 

To minimize   the potential function $\Phi$ via a dynamic programming approach, 
we consider the following 
   lifted control problem:  
for each $t\in [0,T]$ and $\rho = \rho^1 \otimes \dots \otimes \rho^N \in   \cP_2(\sR)^N$,
\begin{equation*}
    \mathcal{W}(t, \rho) \coloneqq \inf_{u \in \mathcal{V}} \mathcal{J}(t, \rho, u),
\end{equation*}
where for each 
  $u=(u^i)_{i\in I}\in \cV$,
  $\mathcal{J}(t, \rho, u)$ is the 
  lifted cost functional given by 
\begin{equation*}
    \mathcal{J}(t,\rho,u) =  \int_t^T \left(  \langle \rho_s, |u(s,x)|^2\rangle + \langle \rho_s, x^\top Q x \rangle \right) \, ds + \langle \rho_T, (x - d)^\top \Lambda (x - d) \rangle.
\end{equation*}
where $\rho:[0,T]\to \cP_2(\sR)^N$
is the marginal laws of the state processes controlled by $u$: 
\begin{equation*}
    d X_s^{i} = u^i(s, X_s^{i}) \, ds + \sigma_s^i \, d B_s^{i}, \quad s \in{[t,T]}; \quad X^i_t \sim \rho^i.
\end{equation*}
By the verification theorem  \cite[Proposition 3.4]{jacksonlacker2025approximately},
suppose that  $\mathcal{U} \in \mathcal{C}^{1,2}([0,T] \times \mathcal{P}_2(\mathbb{R})^N, \mathbb{R})$ satisfies for all $(t,\rho)\in [0,T]\times \cP_2(\sR)^N$,
 \begin{align}
    \label{eq:value_PDE1}
    &- \frac{\partial \mathcal{U}}{\partial t} (t,\rho) - \frac{1}{2} \sum_{i=1}^N 
    (\sigma_t^i)^2
    \langle \rho^i, \frac{\partial}{\partial x} \mathcal{D}_{\rho^i} \mathcal{U}(t,\rho)(x) \rangle + \frac{1}{4} \sum_{i = 1}^N \langle \rho^i,(\mathcal{D}_{\rho^i} \mathcal{U}(t,\rho)(x))^2 \rangle = \langle \rho, x^\top Q x \rangle,
\end{align}
and 
$ \mathcal{U}(T, \rho) = \langle \rho, (x - d)^\top \Lambda (x - d) \rangle$.
Define 
$u^{*,i}(t,x) =-\frac{1}{2} \mathcal{D}_{\rho^i} \mathcal{U}(t, \rho^*_t)(x)$ for   all $i\in I$,
where $\rho^*_t =\mathcal L(X^*_t) $, and $X^*$ satisfies the following McKean-Vlasov SDE:  
\begin{equation*}
    d X_s^{i} =
    -\frac{1}{2} \mathcal{D}_{\rho^i} \mathcal{U}(s, \mathcal L(X_s))(X_s^{i} )\, ds + \sigma_s^i \, d B_s^{i}, \quad s \in{[t,T]}; \quad X^i_t \sim \rho^i.
\end{equation*}
Then $u^*=(u^{*,i})_{i\in I}$ is a minimizer of $\Phi:\cV\to \sR$. 

Now by \cite[Chapter 6, Theorem 7.2]{yong_stochastic_1999},
$\Lambda\in \sS^N_{\ge 0}$
and Assumption \ref{assumption:Q_positivesemidef},
\eqref{eq:P_symmetric} and \eqref{eq:Psi_symmetric}  have  
   unique solutions  $P^i\in \mathcal{C}([0,T], \mathbb{R})$ and 
     $\Psi \in \mathcal{C}([0,T], \mathbb{S}^{N}_{\ge 0})$, respectively.
     Since \eqref{eq:zeta_symmetric} is a linear ODE, 
     it has a unique solution  $\zeta^i\in \mathcal{C}([0,T], \mathbb{R})$. Hence the  system \eqref{eq:ODEs_valuefunction} is well-posed. 
For $i \in I$, let $\varphi^i\in \cC([0,T],\sR)$ satisfy 
    \begin{equation}
        \label{eq:varphi_in_proof}
        \frac{\partial \varphi_t^i}{\partial t} + (\sigma_t^i)^2 P_t^i - (\zeta_t^i)^2 = 0,
        \quad  t \in [0,T];
        \quad \varphi_T^i = (d^i)^2 \gamma^i,
    \end{equation}
 and define for all $(t,\rho) \in [0,T] \times \mathcal{P}_2(\mathbb{R})^N$,
    \begin{equation*}
        \mathcal{U}(t,\rho) \coloneqq \sum_{j=1}^N \left[\varphi_t^j + 2 \zeta_t^j \langle \rho^j, x \rangle + P_t^j (\langle \rho^j, x^2 \rangle - (\langle \rho^j, x \rangle)^2 ) + \sum_{k = 1}^N  \Psi_t^{j,k} \langle \rho^j, x \rangle \langle \rho^k , x \rangle   \right ].
    \end{equation*}
 We claim that $ \mathcal{U}$ satisfies \eqref{eq:value_PDE1}. Indeed, 
by the symmetry of $\Psi$,
    \begin{equation*}
        \begin{split}
            \mathcal{D}_{\rho^i} \mathcal{U}(t, \rho)(x) &= 2 P_t^i (x - \langle \rho^i, y \rangle) + \sum_{k=1}^N \left[ (\Psi_t^{i,k} + \Psi_t^{k,i}) \langle \rho^k, y \rangle \right] + 2 \zeta_t^i \\
            &= 2 P_t^i (x - \langle \rho^i, y \rangle) + 2 [\Psi_t \langle \rho, y \rangle]_i + 2 \zeta_t^i,
        \end{split}
    \end{equation*}
    where $[\cdot]_i$ denotes the $i$-th component of a vector, 
    and $\frac{\partial}{\partial x} \mathcal{D}_{\rho^i} \mathcal{U}(t, \rho)(x) = 2 P_t^i$. Thus 
    \begin{equation*}
        \frac{1}{4} \langle \rho^i, (\mathcal{D}_{\rho^i} \mathcal{U}(t, \rho)(x))^2 \rangle = (P_t^i)^2 ( \langle \rho^i, x^2 \rangle - (\langle \rho^i, x \rangle)^2) + ([\Psi_t \langle \rho, x \rangle]_i)^2 + 2 [\Psi_t \langle \rho, x \rangle]_i \zeta_t^i + (\zeta_t^i)^2.
    \end{equation*}
Using these   expressions, we see  \eqref{eq:value_PDE1} holds if and only if 
    \begin{equation}
        \label{eq:value_PDE2}
        \begin{split}
            \sum_{i = 1}^N \Bigg[ &-\frac{\partial P_t^i}{\partial t} (\langle \rho^i, x^2 \rangle - (\langle \rho^i, x \rangle)^2 ) - \sum_{j = 1}^N \left[ \frac{\partial \Psi_t^{i,j}}{\partial t} \langle \rho^i, x \rangle \langle \rho^j , x \rangle \right] - \frac{\partial \varphi_t^i}{\partial t}  - 2 \frac{\partial \zeta_t^i}{\partial t} \langle \rho^i, x \rangle \\
            &- (\sigma_t^i)^2 P_t^i + (P_t^i)^2 \left( \langle \rho^i, x^2 \rangle - ( \langle \rho^i, x \rangle )^2 \right) + \left([\Psi_t \langle \rho, x \rangle]_i \right)^2 + 2 \left(\sum_{j = 1}^N \zeta_t^j \Psi_t^{j,i} \right) \langle \rho^i, x \rangle  \\
            &+ (\zeta_t^i)^2 \Bigg] = \langle \rho, x^\top Q x \rangle
            =\sum_{i=1}^N \left[ Q_{i,i} \left(\langle \rho^i, x^2 \rangle - (\langle \rho^i, x \rangle )^2 \right) \right] + (\langle \rho, x \rangle)^\top Q \langle \rho, x \rangle,
        \end{split}
    \end{equation}
where the last identity used $\rho = \rho^1 \otimes \dots \otimes \rho^N$.
Since  $ 
        \sum_{i,j = 1}^N    \frac{\partial  \Psi_t^{i,j}}{\partial t} \langle \rho^i, x \rangle \langle \rho^j , x \rangle   = (\langle \rho, x \rangle)^\top \frac{\partial \Psi_t}{\partial t}  \langle \rho, x \rangle$ and $  \sum_{i = 1}^N \left([\Psi_t \langle \rho, y \rangle]_i \right)^2 = (\langle \rho, x \rangle )^\top \Psi_t^\top \Psi_t \langle \rho, x \rangle$, 
 \eqref{eq:value_PDE2} reduces  to
    \begin{equation*}
        \begin{split}
            &\sum_{i=1}^N \Bigg[ \left( -\frac{\partial P_t^i}{\partial t} + (P_t^i)^2 - Q_{i,i} \right) (\langle \rho^i, x^2 \rangle - (\langle \rho^i, x \rangle)^2 ) + 2 \left(-\frac{\partial \zeta_t^i}{\partial t} + \sum_{j = 1}^N \zeta_t^j \Psi_t^{j,i} \right) \langle \rho^i, x \rangle \\
            &+ \left( - \frac{\partial \varphi_t^i}{\partial t} - (\sigma_t^i)^2 P_t^i + (\zeta_t^i)^2 \right) \Bigg] + (\langle \rho, x \rangle)^\top \left( - \frac{\partial \Psi_t}{\partial t} + \Psi_t^\top \Psi_t -Q \right) \langle \rho, x \rangle = 0,
        \end{split}
    \end{equation*}
   which holds using the dynamics of  $(P^i)_{i \in I}, \Psi, (\zeta^i)_{i \in I}$ and $ (\varphi^i)_{i \in I}$. 
The fact that $\mathcal{U}$ satisfies the terminal condition can be verified using the terminal conditions of $(P^i)_{i \in I}, \Psi, (\zeta^i)_{i \in I}$ and $ (\varphi^i)_{i \in I}$. This completes the verification argument, and  proves that $u^*$  defined in \eqref{eq:optimal_policy} is the   minimizer of   $\Phi$ in $\mathcal{V}$, and, if Assumption \ref{assumption:Qij_symmetry} holds,   an NE of $\sG$.

\subsection{Proofs of Propositions 
\ref{prop:potential_decomposition}, 
\ref{prop:graddom_K}
and    \ref{prop:strongconvex_lipschitz_G}}

\begin{proof}[Proof of Proposition \ref{prop:potential_decomposition}]
    Note that for all $t \in [0,T]$,
    \begin{equation*}
        \begin{split}
            \sum_{i=1}^N \mathbb{E}[|K^i_t (X^i_t - \mu^i_t) + G^i_t|^2] &= G_t^\top G_t + \sum_{i=1}^N (K_t^i)^2 \vartheta_t^i, \\
            \mathbb{E}[X_t^\top Q X_t] &= \mathbb{E}[(X_t - \mu_t)^\top \mathrm{diag}( (Q_{i,i})_{i \in I} ) (X_t - \mu_t)] + \mu_t^\top Q \mu_t
            \\
            &=\sum_{i=1}^N Q_{i,i}\vartheta^i_t
            +\mu_t^\top Q \mu_t, \\
            \mathbb{E}[(X_T - d)^\top \Lambda (X_T - d)] &= \mathbb{E}[(X_T - \mu_T)^\top \Lambda (X_T - \mu_T)] + (\mu_T - d)^\top \Lambda (\mu_T - d)
            \\
            &=\sum_{i=1}^N \gamma^i\vartheta^i_T
            +(\mu_T - d)^\top \Lambda (\mu_T - d),
        \end{split}
    \end{equation*}
which proves the decomposition of the potential function. Similar arguments show the decomposition of the cost functionals.
The desired potential structures follow  from these decompositions and 
the fact that $\Phi$ is a potential function for  $\sG=(I,(J^i)_{i\in I}, \cV)$.
\end{proof}

Before proving Proposition \ref{prop:graddom_K}, we establish a cost difference lemma for the functional $\Phi^1$.

\begin{lemma}
    \label{lemma:cost_difference}
    For all  $K \in  \mathcal{K}$ and    $i \in I$, let $P^{\Phi,K,i}\in \cC([0,T],\sR)$ satisfy
    \begin{equation}
        \label{eq:ODE_P_Phi_K}
        \frac{\partial P_t}{\partial t} + 2 K_t^i P_t + (K_t^i)^2 + Q_{i,i} = 0, \quad t \in [0,T]; \quad P_T = \gamma^i.
    \end{equation}
  For all $K,\tilde K\in \cK$,
    \begin{equation*}
        \Phi^1(\Tilde{K}) - \Phi^1(K) = \sum_{i = 1}^N \int_0^T \left( 2 (P_t^{\Phi, K,i} + K_t^i) (\Tilde{K}_t^i - K_t^i)  {\vartheta}_t^{\Tilde K, i} + (\Tilde{K}_t^i - K_t^i)^2 {\vartheta}_t^{\Tilde K, i} \right) \, dt, 
    \end{equation*}
      where ${\vartheta}^{\Tilde K, i}$ is defined in \eqref{eq:moments}. 
    Moreover, $\mathcal{D}_K^{\Phi,i} = 2 ( P^{\Phi,K,i} + K^i )$, with  $\mathcal{D}_K^{\Phi,i}$   defined in \eqref{eq:DK}.
\end{lemma}

\begin{proof}
 The cost difference of $\Phi^1$ follows from similar arguments as those for \cite[Lemmas 3.2 and 3.3]{giegrich2024convergence} and \cite[Lemma 4.2]{plank2025policy}.
 The characterization of $\cD^{\Phi,i}_K$ follows from calculations analogous to those in \cite[Lemma 4.3]{plank2025policy}.
\end{proof}

Using Lemma \ref{lemma:cost_difference}, we  now prove Proposition \ref{prop:graddom_K}. 

\begin{proof}[Proof of Proposition \ref{prop:graddom_K}]

Let  $K^{\Phi,*} = (K^{\Phi,*,i})_{i \in I}$  be  $ K^{\Phi,*,i}\coloneqq - P^i $, where $P^i$ satisfies \eqref{eq:P_symmetric}.
The equation  \eqref{eq:P_symmetric} implies that 
  $P^{\Phi,K^{\Phi,*},i} = P^{i}$, where $P^{\Phi,K^{\Phi,*},i}$ is defined by  \eqref{eq:ODE_P_Phi_K} (with  $K^i=K^{\Phi,*,i}$).
   By Lemma \ref{lemma:cost_difference}, for any $K \in \mathcal{K}$,
    \begin{equation*}
        \Phi^1(K) - \Phi^1(K^{\Phi,*}) = \sum_{i = 1}^N \left[ \int_0^T \left( (K_t^i + P_t^i)^2 \vartheta_t^{K,i} \right) \, dt \right] \ge 0,
    \end{equation*}
    where the equality holds if and only if $K^i = - P^i$ for all ${i \in I}$. Therefore, $K^{\Phi,*}$ is the minimizer of $\Phi^1 : \mathcal{K} \to \mathbb{R}$. The gradient dominance of $\Phi^1$ follows from Lemma \ref{lemma:cost_difference} and a completion-of-squares argument,
    similar to \cite[Lemma~4.5]{plank2025policy}.
\end{proof}

We now state a uniform bound
of the iterates $(K^{(\ell),i})_{\ell \in \mathbb{N}_0, i \in I}$ 
produced by Algorithm \ref{algo}. 
Note that this uniform bound is a property of the gradient update  \eqref{eq:potential_GD_update} for $K$ and holds independently of Assumption \ref{assumption:Qij_symmetry}.

\begin{proposition}
    \label{prop:uniform_bound_K}
    Let $K^{(0)} \in L^\infty([0,T], \mathbb{R}^N)$, 
    $\eta_K^i \in (0, 1/2)$ for all $i\in I$,
    and 
      $(K^{(\ell)})_{\ell \in \mathbb{N}_0}$ be  defined  by 
 \eqref{eq:potential_GD_update}.  
 For all $\ell\in \sN$, 
 let      $P^{(\ell),i} = P^{K^{(\ell)},i}$ be defined by  \eqref{eq:PK_ODE} (with    $K^i = K^{(\ell),i}$).
 Then for all $i \in I$ and  $\ell \in \mathbb{N}_0$, $\lVert K^{(\ell),i} \rVert_{L^\infty} \le \overline{C}_\infty^K$ and $\lVert P^{(\ell),i} \rVert_{L^\infty} \le \overline{C}_\infty^K$, where $\overline{C}_\infty^K>0$ is given by
    \begin{equation}
        \label{eq:C0K}
        \overline{C}_\infty^K \coloneqq \max_{i \in I}{\left(\lVert K^{(0),i} \rVert_{L^\infty} + \lVert P^{(0),i} \rVert_{L^\infty} \right)}.
    \end{equation}

Consequently, 
there exist constants $M^\vartheta, M_\vartheta > 0$ such that  for all $i \in I$, $\ell\in \sN_0$ and $t \in [0,T]$,
$\vartheta_t^{(\ell),i} \in [M_\vartheta, M^\vartheta]$,
where 
$\vartheta^{(\ell),i} = \vartheta^{K^{(\ell)},i}$ satisfies \eqref{eq:moments} (with $K^i = K^{(\ell),i}$).

\end{proposition}
\begin{proof}
Since      the update  \eqref{eq:potential_GD_update} of $ K^{(\ell),i}$ depends only on player $i$'s coefficients, 
    the uniform bounds of $ K^{(\ell),i}$ and 
    $ P^{(\ell),i}$
    can be obtained 
    by adapting the arguments for single-player control problems in  \cite[Proposition 3.5(2)]{giegrich2024convergence}
    to the present setting. 
    The uniform bound of $\vartheta^{(\ell),i}$ 
    follows from 
    \eqref{eq:moments}
    and Gr\"onwall's inequality.
\end{proof}

Finally, we prove the strong convexity and Lipschitz smoothness of $\Phi^2$.
 \begin{proof}[Proof of Proposition \ref{prop:strongconvex_lipschitz_G}]

Define the linear   operator $\Gamma: L^2([0,T], \mathbb{R}^N) \to \cC([0,T], \mathbb{R}^N)$
such that $\Gamma[f](t) \coloneqq \int_0^t f(s) \, ds$ for all   $f \in L^2([0,T], \mathbb{R}^N)$.  By \eqref{eq:moments}, 
  for all $G \in \mathcal{G}$, $\mu^G = \mathbb{E}[\xi] + \Gamma[G]$, and  
\begin{equation*}
    \begin{split}
        \Phi^2(G) &= \int_0^T \left( G_t^\top G_t + (\mathbb{E}[\xi] + \Gamma[G](t))^\top Q (\mathbb{E}[\xi] + \Gamma[G](t))  \right) \, dt \\
        &\quad+ (\mathbb{E}[\xi] + \Gamma[G](T) - d)^\top \Lambda (\mathbb{E}[\xi] + \Gamma[G](T) - d).
    \end{split}
\end{equation*}
It is easy to show that 
the second derivative $\Phi^2$ satisfies for all $G,G_1,G_2\in \cG$,
    \begin{equation*}
        \nabla_G^2 \Phi^2(G)[G_1, G_2] = 2 \left( \langle G_1, G_2 \rangle_{L^2} + \frac{1}{2} \langle \Gamma[G_1], (Q + Q^\top) \Gamma[G_2] \rangle_{L^2} + \langle \Gamma[G_1](T), \Lambda \Gamma[G_2](T) \rangle \right).
    \end{equation*}
  
  By   Assumption \ref{assumption:Q_positivesemidef} and $\gamma^i \ge 0$ for all $i \in I$,   
$  \nabla_G^2 \Phi^2(G)[G, G] \ge 2 \lVert G \rVert_{L^2}^2 = m \lVert G \rVert_{L^2}^2
$       with $m\coloneqq 2$, 
 which implies the     
    strong convexity of $\Phi^2$.
   Moreover, the following estimates of the operator norms  
   $\lVert \Gamma \rVert_{\mathrm{op}, L^2\to L^2} \le T / \sqrt{2}$ 
and  $\lVert \Gamma(T)\rVert_{\mathrm{op}, L^2\to \sR^N} =\sqrt T$ imply that 
    \begin{equation*}
        \nabla_G^2 \Phi^2(G)[G, G] \le ( 2 + T^2 \lambda_{\mathrm{max}}(Q_{\mathrm{sym}}) + 2 T \max_{i \in I} \gamma^i ) \lVert G\rVert_{L^2}^2.
    \end{equation*}
    The Taylor expansion and the quadratic structure of $G \mapsto \Phi^2(G)$ then show that 
    \begin{equation*}
        \Phi^2(\Tilde{G}) - \Phi^2(G) \le \langle \Tilde{G} - G, \nabla_G \Phi^2(G) \rangle_{L^2} + \frac{L}{2} \lVert \Tilde{G} - G \rVert_{L^2}^2,
    \end{equation*}
    where $L \coloneqq 2 + T^2 \lambda_{\mathrm{max}}(Q_{\mathrm{sym}}) + 2 T \max_{i \in I} \gamma^i$.
    This completes the proof. 
\end{proof}

\subsection{Proofs of Theorem \ref{theorem:convergence_potentialcase} and Corollary \ref{cor:approximate_NE_symmetric}}

To simplify the notation,  we write   
  $\eta_{K}^\mathrm{min} = \min_{i \in I} \eta_K^i$, $\eta_{K}^\mathrm{max} = \max_{i \in I} \eta_K^i$,  $\eta_{G}^\mathrm{min} = \min_{i \in I} \eta_G^i$, $\eta_{G}^\mathrm{max} = \max_{i \in I} \eta_G^i.$

We first prove a descent lemma for the function $\Phi^1$. It is essential to impose Assumption \ref{assumption:Qij_symmetry}, which ensures that the gradient of the cost functionals coincides with that of $\Phi^1$. 

\begin{lemma}
    \label{lemma:desc_K}
    Suppose Assumption \ref{assumption:Qij_symmetry} holds. 
    Let $M_\vartheta, M^\vartheta>0$ be the constants from Proposition \ref{prop:uniform_bound_K}, and  $(K^{(\ell)})_{\ell\in \sN_0}$  be defined via \eqref{eq:potential_GD_update}.  If   $\eta_K^i \in (0,M_\vartheta/(2M^\vartheta))$ for all $ i \in I$, then for all $\ell\in \sN_0$,
    \begin{equation*}
        \Phi^1(K^{(\ell + 1)}) - \Phi^1(K^{(\ell)}) \le - \eta_K^{\mathrm{min}} (M_\vartheta - \eta_K^{\mathrm{max}} M^\vartheta) \lVert \mathcal{D}_K^{\Phi,(\ell)} \rVert_{L^2}^2,
    \end{equation*}
    where $\mathcal{D}_{K}^{\Phi,(\ell)} = \mathcal{D}_{K^{(\ell)}}^\Phi$ is defined in \eqref{eq:DK}.
\end{lemma}

\begin{proof}
    Under Assumption \ref{assumption:Qij_symmetry}, for all $i \in I$, $\nabla_{K^i} J^{i}( K, G ) =  \nabla_{K^i} \Phi^{1}( K )$  due to Proposition \ref{prop:potential_decomposition}, hence   $K^{(\ell+1),i} - K^{(\ell),i} = -\eta_K^i \mathcal{D}_{K}^{\Phi,(\ell),i}$. Applying Lemma \ref{lemma:cost_difference} with $K = K^{(\ell)}, \Tilde{K} = K^{(\ell+1)}$ and using the bounds from Proposition \ref{prop:uniform_bound_K}  for $\eta_K^{\mathrm{max}} \in (0, M_\vartheta/(2M^\vartheta))$ yields
    \begin{equation*}
        \begin{split}
            \Phi^1(K^{(\ell + 1)}) - \Phi^1(K^{(\ell)}) &= \sum_{i = 1}^N \int_0^T \left( -\eta_K^i (\vartheta_t^{(\ell + 1),i} - \eta_K^i \vartheta_t^{(\ell + 1),i}) \lvert \mathcal{D}_{K,t}^{\Phi, (\ell),i} \rvert^2 \right) \, dt \\
            &\le  - \eta_K^{\mathrm{min}} (M_\vartheta - \eta_K^{\mathrm{max}} M^\vartheta) \lVert \mathcal{D}_K^{\Phi,(\ell)} \rVert_{L^2}^2.
        \end{split}
    \end{equation*}
    This completes the proof.
\end{proof}

\begin{proof}[Proof of Theorem \ref{theorem:convergence_potentialcase}]
To simplify the notation, 
write $(K^*, G^*)=( K^{\Phi,*}, G^{\Phi,*})$  
defined in Theorem \ref{theorem:optimalpolicy}.
    Define the constants $C_1^K, C_2^K, C_3^K > 0$ by
    \begin{equation}
    \label{eq:CK123}
        C_1^K \coloneqq \frac{M_\vartheta}{2 M^\vartheta}\le  \frac{1}{2}, \quad C_2^K \coloneqq \frac{1}{M_\vartheta}, \quad C_3^K \coloneqq \frac{2 M_\vartheta}{M^{\vartheta,*}},
    \end{equation}
    where $M_\vartheta, M^\vartheta>0$ are the constants from Proposition \ref{prop:uniform_bound_K}   and $M^{\vartheta,*} \coloneqq \sup_{i\in I, t\in [0,T]}\vartheta_t^{K^*,i}$. Note that under Assumption \ref{assumption:Qij_symmetry}, $M_{\vartheta}^* \coloneqq \inf_{i\in I, t\in [0,T]} \vartheta_t^{K^*,i}$ satisfies $M_\vartheta \le M_\vartheta^*$ as $\lVert K^* \rVert_{L^\infty} \le \overline{C}_\infty^K$ using a similar argument as in Proposition \ref{prop:uniform_bound_K}. Hence, $C_1^K \le 1/C_3^K$.  By  Proposition \ref{prop:graddom_K} and Lemma  \ref{lemma:desc_K},
    if $\eta_K^{\mathrm{max}} \le C^K_1$,
  \begin{equation}
  \label{eq:thm_K_conv_argument}
        \begin{split}
            \Phi^1(K^{(\ell +1)}) - \Phi^1(K^{*}) 
            &= \Phi^1(K^{(\ell +1)}) -\Phi^1(K^{(\ell )})+\Phi^1(K^{(\ell )}) -  \Phi^1(K^{*}) 
            \\
            &\le 
            \left(1- 
            \eta_K^{\mathrm{min}} (M_\vartheta - \eta_K^{\mathrm{max}} M^\vartheta)\frac{4}{M^{\vartheta,*}}
\right) \left(\Phi^1(K^{(\ell )}) -  \Phi^1(K^{*}) \right)
\\
&\le 
\left(1- 
            \eta_K^{\mathrm{min}}  \frac{2 M_\vartheta}{M^{\vartheta,*}}
\right) \left( \Phi^1(K^{(\ell )}) -  \Phi^1(K^{*}) \right) .
        \end{split}
    \end{equation}
Applying the above inequality iteratively and  Lemma    \ref{lemma:cost_difference} yield     
    \begin{equation*}
        \begin{split}
      \lVert K^{(\ell +1)} - K^* \rVert_{L^2}^2 
      \le \frac{1}{M_\vartheta}  \left(\Phi^1(K^{(\ell +1)}) - \Phi^1(K^{*}) \right)&\le 
      C_2^K
      \left(1 - \eta_K^{\mathrm{min}} C_3^K \right)^{\ell+1}   \left( \Phi^1(K^{(0)}) - \Phi^1(K^*) \right),
        \end{split}
    \end{equation*}
    which proves the convergence of $(K^{(\ell)})_{\ell \in \sN_0}$. 

    For the convergence of $(G^{(\ell)})_{\ell \in \sN_0}$, 
define $S \coloneqq \mathrm{diag}(\sqrt{\eta_G^1}, \dots, \sqrt{\eta_G^N})$ and
  the transformed cost 
  $\Hat{\Phi}^2(G) \coloneqq \Phi^2(S G)$ for any $G \in \mathcal{G}$.
  The updates of $(G^{(\ell)})_{\ell \in \sN_0}$ 
  and the fact that 
  $\nabla_G \Hat{\Phi}^2(G) =S\nabla \Phi^2(S G)$
  imply that 
$\Hat{G}^{(\ell + 1)} = \Hat{G}^{(\ell)} - \nabla_G \Hat{\Phi}^2(\Hat{G}^{(\ell)})$ for all $\ell\in \sN_0$. By Proposition \ref{prop:strongconvex_lipschitz_G}, $\Hat{\Phi}^2$ is $(\eta_G^{\mathrm{min}} m)$-strongly convex and $(\eta_G^{\mathrm{max}} L)$-Lipschitz smooth. Hence 
similar to  \cite[Theorem 2.1.15]{nesterov2013introductory}, 
if $\eta_G^{\mathrm{max}} \in (0, 1/L)$,
  $ 
        \lVert \Hat{G}^{(\ell + 1)} - \Hat{G}^{*} \rVert_{L^2}^2
        \le (1 - \eta_G^{\mathrm{min}} m)^{\ell + 1} \lVert \Hat{G}^{(0)} - \Hat{G}^{*} \rVert_{L^2}^2,
    $
    where $\Hat{G}^{*} \coloneqq S^{-1} G^{*}$ and 
    $G^{*}$ is the minimizer of $\Phi^2:\cG\to \sR$. 
Consequently, for all $\ell\in \sN_0$, 
    \begin{equation}
    \label{eq:G_error_inducation}
        \begin{split}
           &\lVert G^{(\ell +1)} - G^* \rVert_{L^2}^2 = \lVert S (\Hat{G}^{(\ell + 1)} - \Hat{G}^{*}) \rVert_{L^2}^2 \le \eta_G^{\mathrm{max}} \lVert \Hat{G}^{(\ell + 1)} - \Hat{G}^{*} \rVert_{L^2}^2 \\
           &\le \eta_G^{\mathrm{max}} (1 - \eta_G^{\mathrm{min}} m)^{\ell + 1} \lVert \Hat{G}^{(0)} - \Hat{G}^{*} \rVert_{L^2}^2 = \eta_G^{\mathrm{max}} (1 - \eta_G^{\mathrm{min}} m)^{\ell + 1} \lVert S^{-1} (G^{(0)} - G^{*}) \rVert_{L^2}^2 \\
           &\le \frac{\eta_G^{\mathrm{max}}}{\eta_G^{\mathrm{min}}} (1 - \eta_G^{\mathrm{min}} m)^{\ell + 1} \lVert G^{(0)} - G^* \rVert_{L^2}^2. 
        \end{split}
    \end{equation}
    This finishes the proof.
\end{proof}

\begin{proof}[Proof of Corollary \ref{cor:approximate_NE_symmetric}]

Write $K^*=K^{\Phi,*}$  
and $G^*=G^{\Phi,*}$  
defined in Theorem \ref{theorem:optimalpolicy}.
    By \eqref{eq:thm_K_conv_argument},
    \begin{equation*}
        \Phi^1(K^{(\ell)}) - \Phi^1(K^{*}) \le (1 - \eta_K^{\mathrm{min}} C_3^K)^\ell \left( \Phi^1(K^{(0)}) - \Phi^1(K^{*}) \right) \le \frac{\varepsilon}{2},
    \end{equation*}
    and by Proposition \ref{prop:strongconvex_lipschitz_G} and Theorem \ref{theorem:convergence_potentialcase},
    \begin{equation*}
        \Phi^2(G^{(\ell)}) - \Phi^2(G^{*}) \le \frac{L \eta_G^{\mathrm{max}}}{2 \eta_G^{\mathrm{min}}} \lVert G^{(0)} - G^* \rVert_{L^2}^2 \left( 1 - \eta_G^{\mathrm{min}} m \right)^{\ell} \le \frac{\varepsilon}{2},
    \end{equation*}
    where the last inequalities for both policy parameters hold for $\ell \ge M$ with
    \begin{equation*}
        \begin{split}
            M \coloneqq \Bigg\lceil \max\Bigg(&-\frac{1}{\log{(1 - \eta_K^{\mathrm{min}} C_3^K)}} \log{\left( \frac{2 \left( \Phi^1(K^{(0)}) - \Phi^1(K^{*}) \right)}{\varepsilon} \right)}, \\
            &-\frac{1}{\log{ (1 - \eta_G^{\mathrm{min}} m) }}  \log{\left( \frac{L \eta_G^{\mathrm{max}} \lVert G^{(0)} - G^* \rVert_{L^2}^2}{ \eta_G^{\mathrm{min}} \varepsilon} \right)} \Bigg) \Bigg\rceil.
        \end{split}
    \end{equation*}
    Then by Propositions \ref{prop:Phi_is_alpha_potential} and \ref{prop:potential_decomposition}, for any $i \in I$, $u^i \in \mathcal{V}^i$,
    \begin{equation*}
        \begin{split}
            &J^i(u_{\theta^{(\ell)}}) - J^i(u^i, u_{\theta^{(\ell)}}^{-i}) = \Phi(u_{\theta^{(\ell)}}) - \Phi(u^i, u_{\theta^{(\ell)}}^{-i}) \le \Phi(u_{\theta^{(\ell)}}) - \Phi(u^*) \\
            &= \left(\Phi^1(K^{(\ell)}) - \Phi^1(K^*)\right) + \left(\Phi^2(G^{(\ell)}) - \Phi^2(G^*)\right) \le \varepsilon,
        \end{split}
    \end{equation*}
    where $u^*$, $K^*$ and $G^*$ are   defined in Theorem \ref{theorem:optimalpolicy}.
    This completes the proof. 
\end{proof}

\section{Proofs of Section \ref{section:asymmetric}}

\subsection{Proofs of Theorem \ref{theorem:optimalpolicy_asymmetric}
and Proposition \ref{prop:ODE_system_asym_sol}}

\begin{proof}[Proof of Theorem \ref{theorem:optimalpolicy_asymmetric}]

  An NE of the game $\sG$ can be characterized using a coupled PDE system. Indeed, for each $i \in I$,
  fixing the policy profile $(u^j)_{j\not =i} \in \mathcal{V}^{-i}$, player $i$ 
  considers the following minimization problem (cf.~\eqref{eq:cost_generalsetting}):
    \begin{equation}
    \label{eq:objective_closed_loop}
    \inf_{u^i \in \mathcal{V}^i} \mathbb{E} \left[ \int_0^T \left( \lvert u^i(t,X_t^i) \rvert^2 + \langle \rho_t^{-i}, (\Bar{X}_t^i)^\top Q^i \Bar{X}_t^i \rangle \right) \, dt + \gamma^i \lvert X_T^i - d^i \rvert^2 \right],
    \end{equation}
    where 
     $\rho^{-i}_t=\bigotimes_{j\not =i} \mathcal{L}(X_t^j)$,
     with  $X^j$ satisfying   
       \eqref{eq:state_dynamics} with  $u^j$,
      and $\Bar{X}_t^i \coloneqq (x^1, \dots, x^{i-1}, X_t^i, x^{i+1}, \dots, x^N)$. Above and hereafter, we write $\langle \rho^{-i}, f \rangle$ for the integration of $f:\sR^{N}\to \sR$ with respect to all marginals apart from the $i$-th marginal. 
Treating the flow $(\rho^{-i}_t)_{t\in [0,T]}$ as a time-dependent coefficient
and applying 
the   dynamic programming approach to 
\eqref{eq:objective_closed_loop},
we can characterize 
 the optimal best-response strategy of player $i$ using an HJB equation. Assuming all players take these best-response strategies yields the following sufficient condition for an NE: 
 suppose $V^i \in \mathcal{C}^{1,2}([0,T] \times \mathbb{R}, \mathbb{R})$ and $\rho^{*,i}:[0,T]\to \cP_2(\sR)$, $i \in I$, satisfy for all $(t,x) \in [0,T] \times \mathbb{R}$,
    \begin{equation}
        \label{eq:NE_coupled_system}
        \begin{split}
            &\frac{\partial V^i}{\partial t} (t,x) + \frac{1}{2} (\sigma_t^i)^2 \frac{\partial^2 V^i}{\partial x^2}  (t,x) - \frac{1}{4} \left( \frac{\partial V^i}{\partial x}  (t,x) \right)^2 + F^i(t,x) = 0, \quad V^i(T,x) = \gamma^i \lvert x - d^i \rvert^2, \\
            &F^i(t,x) \coloneqq  \langle \rho_t^{*,-i}, (x^1, \dots, x^{i-1}, x, x^{i+1}, \dots, x^N) Q^i (x^1, \dots, x^{i-1}, x, x^{i+1}, \dots, x^N)^\top \rangle,\\
            &\rho_t^{*,-i} =\bigotimes_{j\not =i}  \rho_t^{*,j}, \quad  \rho_t^{*,j} = \mathcal{L}(X_t^{*,j}),
            \\
            &
            d X_t^{*,i} = - \frac{1}{2} \frac{\partial V^i}{\partial x} (t, X_t^{*,i}) \, dt + \sigma_t^i \, d B_t^i, \quad t \in [0,T]; \quad X_0^{*,i} = \xi^i. 
        \end{split}
    \end{equation}
    Define $u^{*,i}(t,x) \coloneqq  -\frac{1}{2}\frac{\partial V^i}{\partial x} (t,x)$ for all $i \in I$. Then $u^* = (u^{*,i})_{i \in I} \in \mathcal{V}$ is an NE of $\mathbb{G}$.

We now construct a specific solution to \eqref{eq:NE_coupled_system}. 
Given a solution 
$(P^i)_{i\in I}, (\mu^i)_{i\in I}, (\lambda^i)_{i\in I} \in \mathcal{C}([0,T], \mathbb{R}^N)$ to \eqref{eq:ODEs_asymmetric}, 
 for $i \in I$, let $\kappa^i, \vartheta^i\in \cC([0,T],\sR)$  satisfy for all $t \in [0,T]$,
       \begin{align*}
            &\frac{\partial \kappa_t^i}{\partial t} + (\sigma_t^i)^2 P_t^i - (\lambda_t^i)^2 + \sum_{j \in I\setminus \{i\}} Q_{j,j}^i (\vartheta_t^j - (\mu_t^j)^2)  + 2 \sum_{\substack{j \in I\setminus \{i,k\} \\ k \in I\setminus \{i\}}} Q_{j,k}^i \mu_t^j \mu_t^k = 0, \quad \kappa_T^i = \gamma^i (d^i)^2, 
            \\
            &\frac{\partial \vartheta_t^i}{\partial t} = -2 P_t^i \vartheta_t^i + (\sigma_t^i)^2, \quad \vartheta_0^i = \mathbb{V} [\xi^i].
        \end{align*}
    Define $V^i(t,x) \coloneqq P_t^i x^2 + 2 \lambda_t^i x + \kappa_t^i$ for all   $(t,x) \in [0,T] \times \mathbb{R}$.
    The optimal policy $u^{*,i}(t,x) =  -P^i_tx-\lambda^i_t$.
    The function $F^i$ depends on    $\rho^{*,j}$ through the mean 
    $\mathbb{E}[X_t^{*,j}]= \langle \rho_t^{*,j}, x \rangle $
    and the variance 
    $ \mathbb{V}[X_t^{*,j}]=\langle \rho_t^{*,j}, x^2 \rangle - (\langle \rho_t^{*,j}, x \rangle)^2$,
which correspond with 
$\mu^j_t$ and $\vartheta^j_t$, respectively. 
 Using  $(\mu^i)_{i\in I}$ and $(\vartheta^i)_{i\in I}$, 
 $F^i$ in \eqref{eq:NE_coupled_system} reduces to 
    \begin{equation*}
        F^i(t,x) = Q_{i,i}^i x^2 + 2 x \sum_{j \in I\setminus \{i\}} Q_{i,j}^i \mu_t^j + \sum_{j \in I\setminus \{i\}} Q_{j,j}^i (\vartheta_t^j - (\mu_t^j)^2) + 2 \sum_{\substack{j \in I\setminus \{i,k\} \\ k \in I\setminus \{i\}}} Q_{j,k}^i \mu_t^j \mu_t^k.
    \end{equation*}
    Substituting the expressions of $V^i$ and $F^i$ into the HJB equation    \eqref{eq:NE_coupled_system}, 
    we see it suffices to verify 
    for all $(t,x)\in [0,T]\times \sR$,
    \begin{equation*}
        \begin{split}
            &x^2 \left( \frac{\partial P_t^i}{\partial t} - (P_t^i)^2 + Q_{i,i}^i \right) + 2 x \left( \frac{\partial \lambda_t^i}{\partial t} - P_t^i \lambda_t^i + \sum_{j \in I\setminus \{i\}} Q_{i,j}^i \mu_t^j \right) \\
            &+ \frac{\partial \kappa_t^i}{\partial t} + (\sigma_t^i)^2 P_t^i -(\lambda_t^i)^2 + \sum_{j \in I\setminus \{i\}} Q_{j,j}^i (\vartheta_t^j - (\mu_t^j)^2)  + 2 \sum_{\substack{j \in I\setminus \{i,k\} \\ k \in I\setminus \{i\}}} Q_{j,k}^i \mu_t^j \mu_t^k = 0, \\
            &P_T^i = \gamma^i, \quad \lambda_T^i = -\gamma^i d^i, \quad  \kappa_T^i = \gamma^i (d^i)^2,
        \end{split}
    \end{equation*}
    which  holds due to the dynamics of   $(P^i)_{i\in I}, (\mu^i)_{i\in I}, (\lambda^i)_{i\in I}$ and $(\kappa^i)_{i \in I}, (\vartheta^i)_{i \in I}$.
\end{proof}

\begin{proof}[Proof of Proposition \ref{prop:ODE_system_asym_sol}]

   For all $i \in I$, 
    as $Q_{i,i}^i \ge 0$, \eqref{eq:P_asymmetric} 
   has a unique solution 
   $P^i\in \cC([0,T],\sR)$. 
   Hence the well-posedness of 
   \eqref{eq:ODEs_asymmetric}
   reduces to  the well-posedness of the  subsystem \eqref{eq:lambda_asymmetric} and \eqref{eq:ODE_system_asym_mean}.
   Suppose that 
   $(\lambda, \mu)$
   satisfies \eqref{eq:lambda_asymmetric}-\eqref{eq:ODE_system_asym_mean}, 
   differentiating $\mu$ twice with respect to $t$
   and using the system \eqref{eq:ODEs_asymmetric} yield the following boundary value problem (BVP):
    \begin{equation}
        \label{eq:BVP_mu}
        \frac{\partial^2 \mu_t}{\partial t^2}  = \Hat{Q} \mu_t, \quad t \in [0,T]; \quad \mu_0 = \mathbb{E}[\xi], \quad \frac{\partial \mu_T}{\partial t}  = - \Lambda(\mu_T - d).
    \end{equation}
    This  implies that 
    the well-posedness of \eqref{eq:ODEs_asymmetric} is equivalent to that of   \eqref{eq:BVP_mu}.
  
    Now suppose that   
the BVP \eqref{eq:BVP_mu_auxiliary} admits only the  trivial solution. This implies that the BVP \eqref{eq:BVP_mu} has at most one solution. 
  We shall  construct a solution to \eqref{eq:BVP_mu}  through a shooting method. 
 Let $\alpha _{2}\in \mathbb{R}^N$  be a constant to be determined, and  $\mu\in \cC([0,T],  \sR^{N})$
 be    the unique solution to 
  \begin{equation*}
        \frac{\partial^2 \mu_t}{\partial t^2}  = \Hat{Q} \mu_t, \quad t \in [0,T]; \quad \mu_0 = \mathbb{E}[\xi], \quad \frac{\partial \mu_0}{\partial t}  = \alpha_2.
    \end{equation*}
It is known that $\mu_t = Y_t^1 \mathbb{E}[\xi] + Y_t^2 \alpha_2$, 
where  $Y^1, Y^2 \in \mathcal{C}([0,T], \mathbb{R}^{N \times N})$ satisfy 
    \begin{equation*}
        \begin{split}
            &\frac{\partial^2 Y_t^1}{\partial t^2} = \Hat{Q} Y_t^1, \quad t \in [0,T];   \quad Y_0^1 = I_N, \quad \frac{\partial Y_0^1}{\partial t} = \mathbf{0},\\
            &\frac{\partial^2 Y_t^2}{\partial t^2} = \Hat{Q} Y_t^2, \quad t \in [0,T];  \quad Y_0^2 = \mathbf{0}, \quad \frac{\partial Y_0^2}{\partial t} = I_N.
        \end{split}
    \end{equation*}
  To ensure that $\mu$ is a solution   to \eqref{eq:BVP_mu}, it remains to find  $\alpha_2\in \sR^N$ such that $\frac{\partial \mu_T}{\partial t}  =- \Lambda(\mu_T - d)$, or equivalently, 
  $$
  \frac{\partial Y_T^1}{\partial t}  \mathbb{E}[\xi] + \frac{\partial Y_T^2}{\partial t} \alpha_2 = - \Lambda\left(Y_T^1 \mathbb{E}[\xi] + Y_T^2 \alpha_2 - d\right).
  $$
  Such an $\alpha_2$ exists if 
  the matrix 
$ 
        L_T \coloneqq \frac{\partial Y_T^2}{\partial t} + \Lambda Y_T^2
    $  is invertible. 
To show the invertibility of $L_T$, suppose $L_T x = \mathbf{0}$ for some $x \in \mathbb{R}^N$ and define $y_t \coloneqq Y_t^2 x$ for all $t \in [0,T]$. Then $y$ satisfies
    \begin{equation*}
        \frac{\partial^2 y_t}{\partial t^2} = \Hat{Q} y_t, \quad t \in [0,T]; \quad y_0 = \mathbf{0}, \quad \frac{\partial y_T}{\partial t} + \Lambda y_T = L_T x = \mathbf{0}. 
    \end{equation*}
    This implies that $y$ is a solution to  the BVP \eqref{eq:BVP_mu_auxiliary}, which along with the assumption implies that   $y \equiv \mathbf{0}$.
    This along with   
      $\frac{\partial Y^2_0}{\partial t} = I_N$ shows that
     $ 
     \frac{\partial y_0}{\partial t} = I_N x=\mathbf{0}$, and hence   $x = \mathbf{0}$. Thus $L_T$ has a trivial kernel and is invertible. This shows that 
    the BVP \eqref{eq:BVP_mu} has a unique solution and the ODE system \eqref{eq:ODEs_asymmetric}  is well-posed.

       Finally, assume  $\Hat{Q}_{\mathrm{sym}} \in \mathbb{S}^{N}_{\ge 0}$. If $\mu$ satisfies \eqref{eq:BVP_mu_auxiliary},   the integration by parts yields 
    \begin{equation*}
        \int_0^T \left\lvert \frac{\partial \mu_t}{\partial t} \right\rvert^2 \, dt = - \mu_T^\top \Lambda \mu_T - \int_0^T \left( \mu_t^\top \Hat{Q}_{\mathrm{sym}} \mu_t \right) \, dt \le 0,
    \end{equation*}
    where the inequality holds as $\Lambda \in \mathbb{S}^{N}_{\ge 0}$ and $\Hat{Q}_{\mathrm{sym}} \in \mathbb{S}^{N}_{\ge 0}$. This implies that   $\mu =   \mathbf{0}$ and proves that 
    the BVP \eqref{eq:BVP_mu_auxiliary} admits only the trivial solution. 
\end{proof}

\subsection{Proof of Lemmas \ref{lemma:K_gradient_bound} and \ref{lemma:gradient_bound}}

\begin{proof}[Proof of Lemma \ref{lemma:K_gradient_bound}]
    For all $i \in I$, $K^i\in \cK^i$, and $t \in [0,T]$, define $\Delta \phi_t^{K,i} \coloneqq P_t^{K,i} - P_t^{\Phi,K,i}$, where $P^{K,i}$ satisfies \eqref{eq:PK_ODE} and $P^{\Phi,K,i}$ satisfies \eqref{eq:ODE_P_Phi_K} with  $K^i$. Then $\Delta \phi^{K,i}$ satisfies
    \begin{equation*}
        \frac{\partial}{\partial t} \Delta \phi_t^i = 2 K_t^i \Delta \phi_t^i - \frac{1}{2} \sum_{j \in I\setminus \{i\}} (Q_{i,j}^i - Q_{i,j}^j), \quad t \in [0,T]; \quad \Delta \phi_T^i = 0.  
    \end{equation*}
    By the variaton of constants for ODEs, for all $t\in [0,T]$,
    \begin{equation*}
        \Delta \phi_t^{K,i} = -\frac{1}{2} \left( \sum_{j \in I\setminus \{i\}} (Q_{i,j}^i - Q_{i,j}^j) \right) \int_t^T \exp{\left(-2 \int_t^s K_r^i \, dr \right)} \, ds.
    \end{equation*}
    Consider $\phi^{(\ell),i} = \phi^{K^{(\ell)},i}$ for the policy parameter $K^{(\ell),i}$. By Proposition \ref{prop:uniform_bound_K},
    $\|K^{(\ell),i}\|_{L^\infty}\le \overline{C}^K_\infty$,
    \begin{equation*}
        \lvert \Delta \phi_t^{(\ell),i} \rvert \le \frac{C_Q}{4 \overline{C}_\infty^K} \left( \exp{\left( 2 \overline{C}_\infty^K (T - t) \right)} - 1 \right) \le \frac{C_Q}{4 \overline{C}_\infty^K} \exp{\left( 2 \overline{C}_\infty^K (T - t) \right)}.
    \end{equation*}
    Squaring and integrating with respect to $t$ yields,
    \begin{equation*}
        \lVert \Delta \phi^{(\ell),i} \rVert_{L^2}^2 \le \frac{C_Q^2}{4 \cdot 16 (\overline{C}_\infty^K)^3} \left( \exp{\left( 4 \overline{C}_\infty^K T \right)} - 1 \right).
    \end{equation*}
    Finally, by Lemmas \ref{lemma:cost_gradients} and \ref{lemma:cost_difference}, for all $i \in I$,
    \begin{equation*}
        \lVert ( \nabla_{K^i} J^i(K^{(\ell)},G^{(\ell)} ) - \nabla_{K^i} \Phi^1(K^{(\ell)}) )  (\vartheta^{K^{(\ell)},i})^{-1} \rVert_{L^2} = 2 \lVert \Delta \phi^{(\ell),i} \rVert_{L^2} \le \frac{C_Q}{ 4 (\overline{C}_\infty^K)^{3/2}} \exp{\left( 2 \overline{C}_\infty^K T \right)}.
    \end{equation*}
    This concludes the proof.
\end{proof}

\begin{proof}[Proof of Lemma \ref{lemma:gradient_bound}]
By \cite[Corollary 4.11]{carmona_lectures_2016},   for all $i \in I$,  the G\^ateaux derivatives of $\nabla_{G^i}J^{i}=\nabla_{G^i}J^{2,i}$ and $\nabla_{G^i}\Phi^2$ are  given by
    \begin{equation*}
        \begin{split}
            &(\nabla_{G^i} J^{2,i}(G))_t = 2 \left( G_t^i + \int_t^T \left[ Q^i \mu_s \right]_i \, ds + \gamma^i (\mu_T^i - d^i) \right), \\
            &(\nabla_{G^{i}} \Phi^2(G))_t = 2 \left( G_t^i + \frac{1}{2} \int_t^T \left[ (Q + Q^\top) \mu_s\right]_i \, ds + \gamma^i ( \mu_T^i - d^i) \right).
        \end{split}
    \end{equation*}
    Define $W \in \mathbb{R}^{N \times N}$ by
    \begin{equation*}
        W_{i,j} \coloneqq \begin{cases}
          \sum_{l \in I\setminus \{i\}}( {Q_{i,l}^i - Q_{i,l}^l }), & i=j,
           \\
             -(Q_{i,j}^i - Q_{i,j}^j), & i\not = j.
        \end{cases}
    \end{equation*}
    For all $t \in [0,T]$ and $i \in I$,
    \begin{equation*}
        [(Q + Q^\top) \mu_t]_i - 2 [Q^i \mu_t]_i = \sum_{j \in I\setminus \{i\}} (Q_{i,j}^i - Q_{i,j}^j) (\mu_t^i - \mu_t^j) = [W \mu_t]_i.
    \end{equation*}
    Then by the Cauchy-Schwarz   inequality and exchanging the order of integral and sum,
    \begin{equation*}
        \begin{split}
            \sum_{i=1}^N \lvert (\nabla_{G^i} J^{2,i}(G))_t - (\nabla_{G^{i}} \Phi^2(G) )_t \rvert^2 &= \sum_{i=1}^N \left\lvert \int_t^T [W \mu_s]_i \, ds \right\rvert^2  \le T \int_0^T \left\lvert W \mu_s \right\rvert^2 \, ds 
            \\
            &\le T \left\lVert W \right\rVert_{\mathrm{2}}^2 \left\lVert \mu \right\rVert_{L^2}^2  \le 4 T N (M^\mu_{\mathrm{aff}}(\mathfrak B))^2 (C_Q)^2,
        \end{split}
    \end{equation*}
    where the last inequality used
    the definition 
    \eqref{eq:moment_bound_afffine}
    of $M^\mu_{\mathrm{aff}}(\mathfrak B)$ and the following bound of  the spectral norm $\|\cdot\|_2$ (see \cite[Section 5.6, Problem 21]{horn2012matrix}),
    \begin{equation*}
        \left\lVert W \right\rVert_{\mathrm{2}} \le 2 \max_{i \in I} \sum_{j \in I\setminus \{i\}} \lvert Q_{i,j}^i - Q_{i,j}^j \rvert = 2 C_Q.
    \end{equation*}
    This proves the desired inequality
    $ \lVert \nabla_{G} J(K,G) - \nabla_{G} \Phi^2(G)  \rVert_{L^2} \le 2 T \sqrt{N} M^\mu_{\mathrm{aff}}(\mathfrak B ) C_Q $.
\end{proof}

\subsection{Proof of Theorem \ref{theorem:convergence_asymmetric_all}}

To prove Theorem \ref{theorem:convergence_asymmetric_all},
we first quantify the sub-optimality of 
$\Phi^1(K^{(\ell)})-\Phi^1(K^{\Phi,*})$
and 
$\Phi^2(G^{(\ell)})
- \Phi^2(G^{\Phi,*})$.
The following    descent-like lemma will be used for the convergence analysis of
$(K^{(\ell)})_{\ell\in \sN_0}$. 

\begin{lemma}
    \label{eq:descent_lemma_alphapotential}
    Let $K^{(0)}\in L^\infty([0,T],\sR^N)$,
    and 
$\overline{C}_\infty^K, M_\vartheta, M^\vartheta>0$ be the constants   in Proposition \ref{prop:uniform_bound_K}. 
Let 
$\eta_K^i \in (0, M_\vartheta/(2M^\vartheta))$ for all $i \in I$,
and $(K^{(\ell)})_{\ell\in \sN}$ be defined by \eqref{eq:potential_GD_update}.
Then for all  $\ell\in \sN_0$,     \begin{equation*}
        \begin{split}
          &  \Phi^1(K^{(\ell + 1)}) - \Phi^1(K^{(\ell)}) 
          \\
          &\le -\eta_K^{\mathrm{min}} \left( \frac{1}{2} M_\vartheta - \eta_K^{\mathrm{max}} M^\vartheta \right) \lVert \mathcal{D}_{K}^{\Phi,(\ell)} \rVert_{L^2}^2 +   \eta_K^{\mathrm{max}}  \frac{(M^\vartheta)^2+(M_\vartheta)^2}{2 M_\vartheta}    \|\mathcal{D}_{K}^{(\ell)} - \mathcal{D}_{K}^{\Phi,(\ell)}\|^2_{L^2},
        \end{split}
    \end{equation*}
    where  for all $i\in I$ and $t\in [0,T]$,
    $(\mathcal{D}^{\Phi,(\ell),i}_K)_t =  (\nabla_{K^i} \Phi^1(K^{(\ell)}))_t (\vartheta^{K^{(\ell)},i})^{-1} $
    (cf.~\eqref{eq:DK}),
    and 
    $\mathcal{D}_{K}^{(\ell),i} = \nabla_{K^i} J^i(K^{(\ell)},G^{(\ell)} )  (\vartheta^{K^{(\ell)},i})^{-1}$ (cf.~Lemma \ref{lemma:cost_gradients}).
\end{lemma}

\begin{proof}
  For all $i\in I$, $\ell\in \sN_0$, write      $\vartheta^{(\ell+1),i} = \vartheta^{K^{(\ell+1)},i}$. By Lemma \ref{lemma:cost_difference},
    \begin{equation}
    \label{eq:descent_K_asymmetric_step1}
        \begin{split}
            &\Phi^1(K^{(\ell +1)}) - \Phi^1(K^{(\ell)}) \\
            &= \sum_{i = 1}^N \int_0^T \left( 2 (P_t^{\Phi,K^{(\ell)},i} + K_t^{(\ell),i}) (K_t^{(\ell + 1), i} - K_t^{(\ell), i}) \vartheta_t^{(\ell + 1), i} + (K_t^{(\ell + 1), i} - K_t^{(\ell), i})^2 \vartheta_t^{(\ell + 1), i} \right) \, dt \\
            &= \sum_{i = 1}^N \int_0^T \left( \mathcal{D}_{K,t}^{\Phi,(\ell),i} (-\eta_K^i \mathcal{D}_{K,t}^{(\ell),i}) \vartheta_t^{(\ell + 1),i} + (\eta_K^i)^2 \lvert \mathcal{D}_{K,t}^{(\ell),i} \rvert^2 \vartheta_t^{(\ell + 1), i} \right) \, dt.
         \end{split}
         \end{equation}
  We   derive an upper bound of the integrand of \eqref{eq:descent_K_asymmetric_step1}. By  Proposition \ref{prop:uniform_bound_K},
  for all $\eta^i_K\le 1/2$,
  \begin{align*}
  \begin{split}
     & \mathcal{D}_{K,t}^{\Phi,(\ell),i} (-\eta_K^i \mathcal{D}_{K,t}^{(\ell),i}) \vartheta_t^{(\ell + 1),i} + (\eta_K^i)^2 \lvert \mathcal{D}_{K,t}^{(\ell),i} \rvert^2 \vartheta_t^{(\ell + 1), i} 
     \\
     &= 
      \mathcal{D}_{K,t}^{\Phi,(\ell),i} (-\eta_K^i (\mathcal{D}_{K,t}^{(\ell),i} - \mathcal{D}_{K,t}^{\Phi,(\ell),i} + \mathcal{D}_{K,t}^{\Phi,(\ell),i} ) )  \vartheta_t^{(\ell + 1),i} + (\eta_K^i)^2 \lvert \mathcal{D}_{K,t}^{(\ell),i} - \mathcal{D}_{K,t}^{\Phi,(\ell),i} + \mathcal{D}_{K,t}^{\Phi,(\ell),i}  \rvert^2 \vartheta_t^{(\ell + 1),i }
      \\
      &=- \eta_K^i (\vartheta_t^{(\ell + 1),i} - \eta_K^i \vartheta_t^{(\ell + 1),i}) \lvert \mathcal{D}_{K,t}^{\Phi,(\ell),i} \rvert^2
      \\
      &\quad -\eta_K^i \left( (1 - 2 \eta_K^i) (\mathcal{D}_{K,t}^{(\ell),i} - \mathcal{D}_{K,t}^{\Phi,(\ell),i}) \mathcal{D}_{K,t}^{\Phi,(\ell),i} - \eta_K^i \lvert\mathcal{D}_{K,t}^{(\ell),i} - \mathcal{D}_{K,t}^{\Phi,(\ell),i}\rvert^2 \right) \vartheta_t^{(\ell + 1),i}
      \\
      &\le 
      - \eta_K^i (M_\vartheta - \eta_K^i M^\vartheta) \lvert \mathcal{D}_{K,t}^{\Phi,(\ell),i} \rvert^2
      \\
      &\quad +\eta_K^i  \left( (1 - 2 \eta_K^i) |\mathcal{D}_{K,t}^{(\ell),i} - \mathcal{D}_{K,t}^{\Phi,(\ell),i}| |\mathcal{D}_{K,t}^{\Phi,(\ell),i} |+  \eta_K^i  \lvert\mathcal{D}_{K,t}^{(\ell),i} - \mathcal{D}_{K,t}^{\Phi,(\ell),i}\rvert^2 \right) M^\vartheta,
   \end{split}   
  \end{align*}
      which along with the following estimate 
   $$
   (1 - 2 \eta_K^i) |\mathcal{D}_{K,t}^{(\ell),i} - \mathcal{D}_{K,t}^{\Phi,(\ell),i}| |\mathcal{D}_{K,t}^{\Phi,(\ell),i} |
   \le \frac{1}{2}\left(\frac{M_\vartheta}{M^\vartheta}|\mathcal{D}_{K,t}^{\Phi,(\ell),i} |^2 +(1 - 2 \eta_K^i)^2\frac{M^\vartheta}{M_\vartheta}|\mathcal{D}_{K,t}^{(\ell),i} - \mathcal{D}_{K,t}^{\Phi,(\ell),i}|^2\right)
   $$
  shows that if $\eta_K^{\mathrm{max}} \le M_\vartheta/(2M^\vartheta)$,
\begin{align*}
  \begin{split}
     & \mathcal{D}_{K,t}^{\Phi,(\ell),i} (-\eta_K^i \mathcal{D}_{K,t}^{(\ell),i}) \vartheta_t^{(\ell + 1),i} + (\eta_K^i)^2 \lvert \mathcal{D}_{K,t}^{(\ell),i} \rvert^2 \vartheta_t^{(\ell + 1), i} 
     \\
      &\le 
      - \eta_K^i \left(\frac{1}{2}M_\vartheta - \eta_K^i M^\vartheta \right) \lvert \mathcal{D}_{K,t}^{\Phi,(\ell),i} \rvert^2
      \\
      &\quad +  \eta_K^i \left(\frac{1}{2}(1 - 2 \eta_K^i)^2\frac{M^\vartheta}{M_\vartheta} + \eta_K^i  \right)\lvert\mathcal{D}_{K,t}^{(\ell),i} - \mathcal{D}_{K,t}^{\Phi,(\ell),i}\rvert^2   M^\vartheta
      \\
      &\le 
      - \eta_K^{\mathrm{min}} \left(\frac{1}{2}M_\vartheta - \eta_K^{\mathrm{max}} M^\vartheta \right) \lvert \mathcal{D}_{K,t}^{\Phi,(\ell),i} \rvert^2
       +  \eta_K^{\mathrm{max}}  \frac{(M^\vartheta)^2+(M_\vartheta)^2}{2 M_\vartheta}    \lvert\mathcal{D}_{K,t}^{(\ell),i} - \mathcal{D}_{K,t}^{\Phi,(\ell),i}\rvert^2.  
   \end{split}   
  \end{align*}
  This along with \eqref{eq:descent_K_asymmetric_step1}
  yields the desired result.
 
\end{proof}

Using Lemma \ref{eq:descent_lemma_alphapotential},
the following proposition quantifies the sub-optimality   of $(K^{(\ell)})_{\ell\in \sN_0}$.
\begin{proposition}
\label{proposition:convergence_alphapotential_slope}
     Suppose Assumption \ref{assumption:Q_positivesemidef} holds, and let $K^{(0)} \in L^\infty([0,T], \mathbb{R}^N)$. 
     Then there exists $\Bar{\eta} > 0$ such that if the learning rates satisfy $\eta_K^i \in (0, \Bar{\eta})$, $i \in I$, the iterates $(K^{(\ell)})_{\ell \in \mathbb{N}_0}$  generated by  Algorithm~\ref{algo_projected} satisfy the following property:
     for all $\varepsilon > 0$, there exists $M \in \mathbb{N}_0$, depending linearly on $\log{(1/\varepsilon)}$, such that for all $\ell\ge M$,
    \begin{equation*}
        \Phi^1(K^{(\ell)}) - \Phi^1(K^{\Phi,*}) \le \varepsilon + \delta_1(C_Q),
    \end{equation*}
    where 
      $K^{\Phi,*}$ is  the minimizer of $\Phi^1: \mathcal{K} \to \mathbb{R}$ defined in Proposition \ref{prop:graddom_K},
      and
     \begin{equation}
        \label{eq:Bar_varepsilon_1}
        \delta_1(C_Q) \coloneqq 
        \frac{\eta_K^{\mathrm{max}}  }{\eta_K^{\mathrm{min}} C_3^K }\frac{(M^\vartheta)^2+(M_\vartheta)^2}{  M_\vartheta}      \frac{\exp{\left( 4 \overline{C}_\infty^K T \right)}}{ 16 (\overline{C}_\infty^K)^{3}} N(C_Q)^2,
    \end{equation}
    with   $C^K_3$  defined in \eqref{eq:CK123},
    and 
$\overline{C}_\infty^K,   M^\vartheta, M_\vartheta>0$
given  in Proposition \ref{prop:uniform_bound_K}.
\end{proposition}

\begin{proof}
Define 
$\Delta _K \coloneqq \eta_K^{\mathrm{max}}  \frac{(M^\vartheta)^2+(M_\vartheta)^2}{2 M_\vartheta}      \frac{N(C_Q)^2}{ 16 (\overline{C}_\infty^K)^{3}} \exp{\left( 4 \overline{C}_\infty^K T \right)}$.
    Using similar arguments as those for  \eqref{eq:thm_K_conv_argument}, by Proposition \ref{prop:graddom_K} and Lemmas 
        \ref{lemma:K_gradient_bound} and
    \ref{eq:descent_lemma_alphapotential}, if $\eta_K^{\mathrm{max}} \le \min{(C^K_1/2, 2/  C_3^K)}$,
    with  $C^K_1$  defined in \eqref{eq:CK123},
    $ \frac{1}{2} M_\vartheta - \eta_K^{\mathrm{max}} M^\vartheta\ge \frac{1}{2} M_\vartheta - \frac{C^K_1}{2} M^\vartheta = \frac{M_\vartheta}{4} $, and  
    \begin{equation*}
        \begin{split}
            \Phi^1(K^{(\ell)}) - \Phi^1(K^{\Phi,*})  
&\le \left(1- 
           \eta_K^{\mathrm{min}} \left( \frac{1}{2} M_\vartheta - \eta_K^{\mathrm{max}} M^\vartheta \right) \frac{4}{M^{\vartheta,*}}
\right)  \left(\Phi^1(K^{(\ell-1)}) - \Phi^1(K^{\Phi,*}) \right) + \Delta _K
            \\
            &\le \left(1 - \frac{1}{2} \eta_K^{\mathrm{min}} C_3^K \right) \left(\Phi^1(K^{(\ell-1)}) - \Phi^1(K^{\Phi,*}) \right) + \Delta _K  \\
            &\le \left(1 - \frac{1}{2} \eta_K^{\mathrm{min}} C_3^K \right)^{\ell} \left(\Phi^1(K^{(0)}) - \Phi^1(K^{\Phi,*}) \right) + \frac{2 \Delta_K}{\eta_K^{\mathrm{min}} C_3^K}  \\
            &\le \varepsilon + \delta_1(C_Q),
        \end{split}
    \end{equation*}
    provided that $\ell \ge M\coloneqq \left\lceil -\frac{1}{\log{(1 -   \eta_K^{\mathrm{min}} C_3^K/2)}} \log{\left(\frac{\Phi^1(K^{(0)}) - \Phi^1(K^{\Phi,*}) }{\varepsilon} \right)}\right\rceil
   $. This completes the proof.
\end{proof}

We then quantify the sub-optimality of $(G^{(\ell)})_{\ell\in \sN_0}$.

\begin{proposition} 
    \label{proposition:convergence_alpha_potential_projection}
    Suppose Assumptions \ref{assumption:Q_positivesemidef} and \ref{assumption:policy_class} hold, and let $m, L > 0$ be given in Proposition \ref{prop:strongconvex_lipschitz_G}.  For all
     $G^{(0)}=(G^{(0),i})_{i\in I} \in \mathcal{G}$
     with $\max_{i\in I}\|G^{(0),i}\|_{L^2}\le \overline{C}_G$, and all
    learning rates $\eta_G^i \in (0, 1/L)$, $i \in I$, satisfying $\eta_G^{\mathrm{min}} > \eta_G^{\mathrm{max}} / (1 + m \eta_G^{\mathrm{max}})$, the iterates $(G^{(\ell)})_{\ell \in \mathbb{N}_0}$  generated by Algorithm~\ref{algo_projected} satisfy the following property: for all $\varepsilon > 0$, there exists $M \in \mathbb{N}_0$, depending linearly on $\log{(1/\varepsilon)}$, such that for all $\ell\ge M$,
    \begin{equation*}
        \Phi^{2}(G^{(\ell)}) - \Phi^{2}(G^{\Phi,*}) \le \varepsilon + \delta_2(C_Q),
    \end{equation*}
    where 
      $G^{\Phi,*}$ is the minimizer of $\Phi^2: \mathcal{G} \to \mathbb{R}$, and 
    \begin{equation}
        \label{eq:Bar_varepsilon_2_P}
        \delta_2(C_Q) \coloneqq 2L \left( \frac{  \eta_G^{\mathrm{max}} T M^\mu_{\mathrm{aff}}(\overline{C}_G)}{1 - \sqrt{\eta_G^{\mathrm{max}} (1 - \eta_G^{\mathrm{min}} m) / \eta_G^{\mathrm{min}} }}  \right)^2 N(C_Q)^2.
    \end{equation}

\end{proposition}

\begin{proof}
    Under Assumption \ref{assumption:Q_positivesemidef}, 
    by Proposition \ref{prop:strongconvex_lipschitz_G},
    $  \Phi^2$ is $m$-strongly convex and $L$-Lipschitz smooth.  Hence  for $\eta_G^i \in (0, 1/L)$ for all $i \in I$, similar arguments to those used for \eqref{eq:G_error_inducation} yield
    \begin{equation}
        \label{eq:helper_graddesc}
        \lVert G^{(\ell)} - \eta_G \odot \nabla_G \Phi^2(G^{(\ell)}) - G^{\Phi,*} \rVert_{L^2}^2 \le \frac{\eta_G^{\mathrm{max}}}{\eta_G^{\mathrm{min}}} (1 - \eta_G^{\mathrm{min}} m) \lVert G^{(\ell)} - G^{\Phi,*} \rVert_{L^2}^2,
    \end{equation}
    where $\odot$ denotes componentwise multiplication.
    The conditions on $\eta_G^{\mathrm{max}}$ and 
    $\eta_G^{\mathrm{min}}$ imply that 
$\eta_G^{\mathrm{max}} (1 - \eta_G^{\mathrm{min}} m) / \eta_G^{\mathrm{min}}\in (0,1)$. Moreover, $G^{\Phi,*,i} = \mathcal{P}_{\overline{C}_G}(G^{\Phi,*,i})$ by Assumption \ref{assumption:policy_class}.
Let 
$\Delta_G= 2 T \sqrt{N} C_Q M^\mu_{\mathrm{aff}}(\overline{C}_G)$.
By the non-expansiveness of 
 the projection operator $\mathcal{P}_{\overline{C}_G}$ and  Lemma \ref{lemma:gradient_bound}, 
    \begin{equation*}
        \begin{split}
            \lVert G^{(\ell + 1)} - G^{\Phi,*} \rVert_{L^2} &= \left( \sum_{i = 1}^N \lVert \mathcal{P}_{\overline{C}_G}\left( G^{(\ell),i} - \eta_G^i \nabla_{G^i} J^i(K^{(\ell)}, G^{(\ell)} ) \right) - G^{\Phi,*,i} \rVert_{L^2}^2 \right)^{1/2} \\
            &\le \left( \sum_{i = 1}^N \lVert G^{(\ell),i} - \eta_G^i \nabla_{G^i} J^i(K^{(\ell)}, G^{(\ell)} ) - G^{\Phi,*,i} \rVert_{L^2}^2 \right)^{1/2} \\
            &= \lVert G^{(\ell)} - \eta_G \odot \nabla_G J(K^{(\ell)}, G^{(\ell)})  - G^{\Phi,*} \rVert_{L^2} \\
            &\le \lVert G^{(\ell)} - \eta_G \odot \nabla_G \Phi^2(G^{(\ell)})  - G^{\Phi,*} \rVert_{L^2} + \lVert \eta_G \odot (\nabla_{G} J(K^{(\ell)},G^{(\ell)}) - \nabla_{G} \Phi^2(G^{(\ell)})) \rVert_{L^2} \\
            &\le \left( \frac{\eta_G^{\mathrm{max}}}{\eta_G^{\mathrm{min}}} (1 - \eta_G^{\mathrm{min}} m)\right)^{1/2} \lVert G^{(\ell)} - G^{\Phi,*} \rVert_{L^2} + \eta_G^{\mathrm{max}} \Delta_G \\
            &\le \left( \frac{\eta_G^{\mathrm{max}}}{\eta_G^{\mathrm{min}}} (1 - \eta_G^{\mathrm{min}} m)\right)^{(\ell + 1)/2} \lVert G^{(0)} - G^{\Phi,*} \rVert_{L^2} + \frac{\eta_G^{\mathrm{max}} \Delta_G}{1 - \sqrt{\eta_G^{\mathrm{max}} (1 - \eta_G^{\mathrm{min}} m) / \eta_G^{\mathrm{min}} }} \\
            &= \left( \frac{\eta_G^{\mathrm{max}}}{\eta_G^{\mathrm{min}}} (1 - \eta_G^{\mathrm{min}} m)\right)^{(\ell + 1)/2} \lVert G^{(0)} - G^{\Phi,*} \rVert_{L^2} + \Hat{\delta}_2,
        \end{split}
    \end{equation*}
    where 
     $\Hat{\delta}_2 \coloneqq (\eta_G^{\mathrm{max}} \Delta_G) / (1 - \sqrt{\eta_G^{\mathrm{max}} (1 - \eta_G^{\mathrm{min}} m) / \eta_G^{\mathrm{min}}})$. This along with  Proposition \ref{prop:strongconvex_lipschitz_G} imply that 
    \begin{equation*}
        \begin{split}
            \Phi^{2}(G^{(\ell)}) &- \Phi^{2}(G^{\Phi,*}) \le \frac{L}{2} \lVert G^{(\ell)} - G^{\Phi,*} \rVert_{L^2}^2 \\
            &\le \frac{L}{2} \left(  \left( \frac{\eta_G^{\mathrm{max}}}{\eta_G^{\mathrm{min}}} (1 - \eta_G^{\mathrm{min}} m)\right)^{\ell/2} \lVert G^{(0)} - G^{\Phi,*} \rVert_{L^2} + \Hat{\delta}_2 \right)^2 \\
            &\le \varepsilon + \frac{L}{2} (\Hat{\delta}_2)^2 = \varepsilon + \delta_2(C_Q),
        \end{split}
    \end{equation*}
    where the last inequality holds for all $\ell \ge M$ with
    \begin{equation*}
        \begin{split}
            M &\coloneqq \left\lceil -\frac{\max \Bigg( \log{\left(\frac{L \lVert G^{(0)} - G^{\Phi,*} \rVert_{L^2}^2}{\varepsilon} \right)}, 2\log{\left( \frac{2 L \lVert G^{(0)} - G^{\Phi,*} \rVert_{L^2} \Hat{\delta}_2 }{\varepsilon} \right)} \Bigg)}{\log{(\eta_G^{\mathrm{max}}(1 - \eta_G^{\mathrm{min}} m )/\eta_G^{\mathrm{min}})}} \right\rceil.
        \end{split}
    \end{equation*}
    This concludes the proof.
\end{proof}

We are now ready to prove 
Theorem \ref{theorem:convergence_asymmetric_all}. 

\begin{proof}[Proof of Theorem \ref{theorem:convergence_asymmetric_all}]
For all $\ell\in \sN$, $u_{\theta^{(\ell)}} \in  \mathcal{V}_{\mathrm{b}}$. 
Let $u_{\theta^{\Phi,*}}$ be the minimizer of $\Phi: \mathcal{V} \to \mathbb{R}$ given in \eqref{eq:optimal_policy}, 
with    the policy parameters $\theta^{\Phi,*} = (K^{\Phi,*}, G^{\Phi,*})$.
By Proposition \ref{prop:Phi_is_alpha_potential},
$\Phi$ is an $\alpha$-potential function of the game $\sG=(I,(J^i)_{i\in I},\cV_{\rm b})$, with $\alpha \coloneqq  \left( M_{\mathrm{b}}^\vartheta + 3 (M_{\mathrm{b}}^\mu)^2 \right) C_Q$.
Moreover, 
under the current assumptions, Propositions~\ref{proposition:convergence_alphapotential_slope} and~\ref{proposition:convergence_alpha_potential_projection} apply.
Thus
for all $\varepsilon>0$,
there exists $M\in \sN$,
such that for all $\ell\ge M$,
$i \in I$ and $u^i \in \mathcal{V}_{\mathrm{b}}^i$, 
    \begin{equation*}
        \begin{split}
            J^i(u_{\theta^{(\ell)}}) &- J^i(u^i, u_{\theta^{(\ell)}}^{-i}) \le \Phi(u_{\theta^{(\ell)}}) - \Phi(u^i, u_{\theta^{(\ell)}}^{-i}) + \alpha \le \Phi(u_{\theta^{(\ell)}}) - \Phi(u_{\theta^{\Phi,*}}) + \alpha \\
            &= \Phi(K^{(\ell)}, G^{(\ell)}) - \Phi(K^{\Phi,*}, G^{\Phi,*}) + \alpha \\
            &= \left(\Phi^{1}(K^{(\ell)}) - \Phi^{1}(K^{\Phi,*}) \right) + \left(\Phi^{2}(G^{(\ell)}) - \Phi^{2}(G^{\Phi,*}) \right) + \alpha
            \le \varepsilon + \delta(C_Q),
        \end{split}
    \end{equation*}
    where the last inequality used Propositions \ref{proposition:convergence_alphapotential_slope} and \ref{proposition:convergence_alpha_potential_projection}, and 
    \begin{equation}
        \label{eq:delta}
        \delta(C_Q) \coloneqq \delta_1(C_Q) + \delta_2(C_Q) + \alpha,
    \end{equation}
    where $\delta_1(C_Q)$ and  $\delta_2(C_Q)$ 
    are defined in 
\eqref{eq:Bar_varepsilon_1}
and \eqref{eq:Bar_varepsilon_2_P}, respectively. 
This finishes the proof. 
\end{proof}

\section*{Acknowledgments}
PP is supported by the Roth Scholarship by Imperial College London and the Excellence Scholarship by Gesellschaft f\"ur Forschungsf\"orderung Nieder\"osterreich (a subsidiary of the province of Lower Austria).
YZ is  partially supported by a grant from the Simons Foundation, and  funded in part by JPMorgan Chase \& Co. Any views or opinions expressed herein are solely those of the authors listed, and may differ from the views and opinions expressed by JPMorgan Chase \& Co.~or its affiliates. This material is not a product of the Research Department of J.P.~Morgan Securities LLC. This material does not constitute a solicitation or offer in any jurisdiction.
 
The authors would like to thank the Isaac Newton Institute for Mathematical Sciences, Cambridge, for support and hospitality during the programme Bridging Stochastic Control And Reinforcement Learning, where work on this paper was undertaken. 
This work was supported by EPSRC grant EP/V521929/1.

\bibliographystyle{plain}
\bibliography{literature.bib}

\end{document}